\documentclass[a4paper,12pt]{amsart}
\usepackage{geometry,graphicx}

\usepackage{amsmath,amsfonts,amssymb,latexsym}
\usepackage{epstopdf}
\usepackage[kerning=true]{microtype} 
\usepackage{graphicx}
\usepackage{wrapfig}

\usepackage{color}

\theoremstyle{theorem}

\newtheorem{theorem}{Theorem}

\newtheorem{proposition}{Proposition}

\newtheorem{lemma}{Lemma}

\newtheorem{corollary}{Corollary}

\theoremstyle{definition}

\newtheorem{definition}{Definition}

\newtheorem{example}{Example}

\newtheorem{remark}{Remark}

\newcommand{\eps}{{\varepsilon}}

\newcommand{\sys}{{\rm sys}}

\newcommand{\Sys}{\sigma}

\newcommand{\SSys}{{\mathfrak S}}

\newcommand{\vol}{{\rm vol}}

\newcommand{\R}{{\mathbb R}}

\newcommand{\N}{{\mathbb N}}

\newcommand{\TT}{{\mathbb T}}

\newcommand{\Z}{{\mathbb Z}}

\newcommand{\C}{{\mathbb C}}

\newcommand{\G}{{\mathcal G}}

\newcommand{\U}{{\mathcal U}}

\newcommand{\area}{{\rm area}}

\newcommand{\A}{\mathcal A}

\numberwithin{equation}{section}
\numberwithin{remark}{section}
\numberwithin{theorem}{section}
\numberwithin{corollary}{section}
\numberwithin{definition}{section}
\numberwithin{lemma}{section}
\numberwithin{claim}{section}
\numberwithin{proposition}{section}

\title{Systolic geometry and simplicial complexity for groups}

\author{I.~Babenko, F.~Balacheff and G.~Bulteau}

\thanks{This work was partially supported by the grant RFSF 10-01-00257-a and the program ANR Finsler-12-BS01-0009-02}

\address{I. Babenko, Institut Montpelli\'erain Alexander Grothendieck,  Universit\'e Montpellier 2, France}

\email{ivan.babenko@umontpellier.fr}

\address{F. Balacheff, Laboratoire Paul Painlev\'e, Universit\'e Lille 1, France.}

\email{florent.balacheff@math.univ-lille1.fr}

\address{G. Bulteau, Institut Montpelli\'erain Alexander Grothendieck, Universit\'e Montpellier 2, France}

\email{guillaume.bulteau@ac-montpellier.fr}



\begin{document}%
\clearpage
\begin{abstract}
Twenty years ago Gromov asked about how large is the set of isomorphism classes of groups whose systolic area is  bounded from above. This article introduces a new combinatorial invariant for finitely presentable groups called {\it simplicial complexity} that allows to obtain a quite satisfactory answer to his question. Using this new complexity, we also derive new results on systolic area for groups that specify its topological behaviour.
\end{abstract}

\maketitle


\section{Introduction}

We focus on groups which can be presented as the fundamental group of a finite $2$-dimensional
simplicial complex. Thus, throughout this article, by {\it group} we mean {\it finitely presentable group},
and by {\it simplicial complex} we mean {\it finite simplicial complex}.\\

First recall the definition of systolic area for $2$-dimensional
simplicial complexes and  groups. Let $X$ be a simplicial complex
of dimension $2$ and suppose that its fundamental group is not
trivial. Given a piecewise smooth Riemannian metric $h$ on $X$ the
systole denoted by $\sys(X,h)$ is defined as the shortest length
of a non-contractible closed curve in $X$. We call {\it systolic
area} the number
 $$
 \Sys(X):=\inf_h \; {\area(X,h) \over \sys(X,h)^2}
 $$
where the infimum is taken over all piecewise smooth Riemannian metrics on $X$.
Following \cite[p.337]{Gro96} the {\it systolic area} of a group $G$ is the number
$$
\Sys(G):=\inf_X \; \Sys(X)
$$
where the infimum is taken over all $2$-dimensional simplicial complexes $X$ with fundamental group $G$.

\medskip

It is straightforward that systolic area for free groups is zero. For a non-free group $G$, we have the following universal lower bound:
$$
\Sys(G)\ge \dfrac{\pi}{16}.
$$
This bound was proved by Rudyak \& Sabourau in \cite{RS08} where the authors also posed
the following two fundamental questions: given a non-free group $G$, is it true that
\begin{enumerate}
\item $\Sys(G)\geq{2/\pi}$ ?

\medskip

\item $\Sys(G\ast\Z)=\Sys(G)$ ?
\end{enumerate}

\medskip

The first question is motivated by the fact that systolic area for surfaces is known to be minimal
for $\R P^2$ whose corresponding fundamental group is  $\Z_2$ and that its value is precisely $2/\pi$
(see \cite[5.2.B]{Gro83} and \cite{Pu52}).\\

 The second question is related to finitude problems for systolic area.
 It is easy to check that systolic area is subadditive for free products, and that in particular
\begin{equation}\label{eq:facteurs.libres}
\Sys(G\ast F_n)\leq \Sys(G)
\end{equation}
where $F_n$ denotes the free group of rank $n$. In \cite[p.337]{Gro96}
Gromov raised the following question: given a positive number $T$ how large is the set
of isomorphism classes of groups with systolic area at most $T$ ? Because of inequality
(\ref{eq:facteurs.libres}) we may hope a finitude result only if we consider groups without
$\Z$ as a free factor (or at least with an uniformly bounded number of such factors). In \cite{RS08}
the authors proved such a finitude result and give a (non-sharp) upper bound for the cardinality.
We give the exact statement of their result in the next paragraph, but first observe that
an affirmative answer to the second question above would thus imply that the set of values
of the systolic area function lying in a compact interval is always finite.\\

Let us now make precise the notion of groups without $\Z$ as a free factor.
According to  \cite[\S 35]{Kur60}---see also  \cite{Mas67}
for a topological version--- for any group $G$ there exist an unique integer $n$
and an unique subgroup $H$ (up to conjugation) such that $G$ decomposes into a free product
\begin{equation} \label{eq:facteurs.libres.groupe}
G = H \ast F_n
\end{equation}
where $H$ can not be decomposed in its turn like in (\ref{eq:facteurs.libres.groupe}) with a positive $n$.
 We call this number $n$ the {\it free index} of $G$. A group without $\Z$ as a free factor is thus a group of free index zero.\\

Denote by $\G_\Sys(T)$ the set of isomorphism classes of groups $G$ with free index zero such that $\Sys(G)\leq T$.
According to \cite{RS08} it is a finite set whose cardinality satisfies the upper bound
\begin{equation}\label{eq:finitudeRS}
\vert \G_\Sys(T)\vert \leq A^{T^3}
\end{equation}
for some explicit constant  $A>1$ and the lower bound
\begin{equation}\label{eq:infRS}
2^T \leq \vert \G_\Sys(T)\vert
\end{equation}
for $T$ large enough.
In this article we propose an alternative proof of the finitude of $\G_\Sys(T)$ which leads
to an improvement of the upper bound (\ref{eq:finitudeRS}). We also give a lower bound
for the cardinality of the subset $\A_\Sys(T) \subset \G_\Sys(T)$ consisting of isomorphism classes of {\it finite abelian} groups.

\begin{theorem}\label{th:finitude}
There exist constants $B$, $B'$ and $B''$ such that the following inequality holds for every $T \geq 2$
$$
 \left[2^{{\pi \over 1+2\sqrt{3}}T}\right] \leq |\mathcal{A}_{\Sys}(T)| \leq |\mathcal{G}_{\Sys}(T)| \leq B^{T^{1 + {B' \over \sqrt{\ln (B'' T)}}}}.
$$
Here $[x]$ denotes the integral part of a number $x$.
\end{theorem}
We thus derive the following answer to Gromov's original question.
\begin{corollary}
We have 
$$
\vert \G_\Sys(T) \vert \leq B^{T^{1 + \varepsilon}}
$$
for any positive $\eps$ provided $T$ is large enough. In particular,
$$
\mathop{\lim}\limits_{T \rightarrow
\infty}{\ln\ln|\mathcal{G}_{\Sys}(T)| \over \ln T} = 1.
$$
\end{corollary}
Observe that our lower bound in Theorem \ref{th:finitude} does not improve inequality (\ref{eq:infRS}), as ${\pi /(1+2\sqrt{3})}\simeq 0.7$.
But it shows that an exponential asymptotic growth is already realized on the class
of finite abelian groups.

The main tool to improve the upper bound of Rudyak \& Sabourau is a
new combinatorial invariant for groups
introduced in section \ref{sec:simplicialcomplexity}. Namely, given a group $G$, this new invariant called {\it simplicial
complexity} and denoted by $\kappa(G)$ is the minimal number of
$2$-simplices of a $2$-dimensional simplicial complex with
fundamental group $G$ (see Definition
\ref{def:simplicialcomplexity}). This invariant might be thought as a discrete version of area for groups.

For groups with free index zero the two invariants $\Sys(G)$ and $\kappa(G)$ are closely related.
The central result of this article is the following comparison theorem.

\begin{theorem}\label{th:syscomplexity}
Let $G$ be a group with free index zero. Then
$$
2\pi \sigma (G) \leq \kappa (G) \leq 625 (50^2 \cdot \sigma(G))^{1 + {2(1 + \ln 5)  \over \sqrt{\ln(50^2 \cdot  \sigma(G))}}}.
$$
\end{theorem}

 This theorem is the link between the different results in this paper and will be proven in section \ref{sec:complexityvssystolicarea}.
It shows that for large values the systolic area $\sigma$ is a quasi-linear function of the simplicial complexity $\kappa$. More precisely, for any positive $\varepsilon > 0$
$$
2\pi \Sys (G) \leq \kappa (G) \leq (\Sys (G))^{1 + \varepsilon}
$$
provided $\kappa(G)$ is large enough. Observe that there is no hope for a linear upper bound as for surface groups it is known that systolic area is strictly sublinear (see \cite{BS94,BPS12}) while the simplicial complexity is linear in terms of the genus (see Example \ref{ex:surfaces} in section \ref{sec:def}).\\

Using Theorem \ref{th:syscomplexity} the problem of estimating $|\mathcal{G}_{\Sys}(T)|$ transforms into a purely combinatorial problem. For a positive integer
$T$ we denote  by $\mathcal{G}_{\kappa}(T)$ the set of isomorphism classes of groups $G$ with free index zero such that $\kappa(G) \leq T$. We also denote by $\mathcal{A}_{\kappa}(T)$ the subset corresponding to finite abelian groups.

\vskip7pt
\begin{theorem}\label{th:borne.sup}
For any $T\geq 2$
$$
\left[2^{T-3 \over 14}\right] \leq |\mathcal{A}_{\kappa}(T)| \leq |\mathcal{G}_{\kappa}(T)| \leq 2^{6T\log_2T}.
$$
\end{theorem}
\vskip7pt

The lower bound is proved by estimating the simplicial complexity of $\Z_m$. For the upper bound we code simplicial complexes with a minimal number of $2$-simplices by some special colored graphs and estimate their number, see section \ref{sec:finitude}.\\

Here is another application of simplicial complexity. In \cite{BPS12} is proved that for any group $G$
$$
\Sys(G)\geq C \, \frac{b_1(G) +1}{(\ln (b_1(G)+2))^2}
$$
for some universal constant $C$ where $b_1(G)$ denotes the first real Betti number of $G$. But this lower bound is inefficient for groups whose first integral homology group has large torsion, such as $\Z_m$ when $m$ is large. In converse simplicial complexity is quite sensitive to torsion elements (see Proposition  \ref{prop:simplexe2}), and using correspondence of Theorem \ref{th:syscomplexity} we are able to prove that for any positive $\varepsilon$
$$
\Sys(G) \geq (\ln |\mbox{Tors}\,H_1(G, \Z)|)^{1-\varepsilon}
$$
for groups with large torsion in homology. In Theorems \ref{teo:sigma.abelien} and \ref{teo:sigma.abelien2} we complete the study of $\sigma(G)$ for abelian groups in terms of two parameters: the number of elements in $G$ and the number of its invariant factors (which coincides with its minimal number of generators). Using this estimate we conclude that
\begin{equation}\label{eq:syscyclique}
(\log_2 m)^{1 - \varepsilon} \leq \Sys (\Z_m) \leq 1.43 \, \log_2 m
\end{equation}
for any positive $\varepsilon$ provided $m$ is large enough. In comparison, inequality (\ref{eq:finitudeRS}) implies that $\Sys(\Z_m)\to \infty$ for large values of $m$ but gives no information about the asymptotic behaviour of this sequence.\\

In literature can be found two other invariants that measure the complexity of finitely presented groups: the $T$-invariant of Delzant \cite{Del96} and the $c$-complexity of Matveev \& Pervova \cite{MP01,PP08}. We compare in subsection \ref{subsec:comparison} simplicial complexity with both  $T$-invariant and $c$-complexity. 

Delzant's $T$-invariant is additive for free product, this nice property being the main reason of its introduction. Nevertheless it is not sensitive to $2$-torsion which makes it not pertinent for systolic considerations on groups.

The $c$-complexity was introduced to measure the complexity of $3$-manifolds using their fundamental group. Despite the fact that simplicial complexity and $c$-complexity agree up to some universal constants (see Proposition \ref{prop:c-kappa.comparaison}), they are of a different kind: simplicial complexity is topological, while $c$-complexity is algebraical which makes it not suitable for systolic geometry. Indeed our central result is the comparison Theorem \ref{th:syscomplexity} and its proof reveals a natural connection between systolic area and simplicial complexity. 
Furthermore we reprove for simplicial complexity some analogs of classical results on $c$-complexity, like the lower bound in terms of the torsion (Proposition \ref{prop:simplexe2}) and the estimate for abelian groups (Theorem \ref{teo:cyclique}). Besides the fact that these results give an alternative proof of the corresponding results for $c$-complexity (compare with \cite{PP08}), they are always more precise than if we simply have used the corresponding results for $c$-complexity and the linear equivalence with simplicial complexity.

So, besides its own interest as a simple and concrete invariant associated to a group, simplicial complexity is shown here to have natural and strong geometric applications.\\

In the last section we present some applications of simplicial complexity to systolic geometry of higher dimensional spaces. One key result is the following estimate for the systolic volume $\SSys$ of any $(2n+1)$-dimensional lens space $L_m^{2n+1}$ with fundamental group $\Z_m$ (see section \ref{sec:sysvol} for definitions).

\begin{theorem}
There exists positive constants $C_n$, $C'_n$ and $D_n$ depending only on the dimension $n$ such that for any integer $m\geq 2$ 
\begin{equation}\label{eq:sigma(L_n(m))}
C_n\left(\ln m\right)^{1 - \frac{C'_n}{\sqrt{\ln \ln m}}} \leq \SSys(L_m^{2n+1}) \leq D_n \, m^n.
\end{equation}
\end{theorem}
While the lower bound is of the same kind as that for $\Sys(\Z_m)$, the best known upper bound is thus polynomial of degree $n$ in $m$. Observe that this degree is half of the degree in the trivial upper bound $\approx m^{2n}$ given by the round metric. In particular round metrics on lens spaces are not systolically extremal for large $m$.
Determining the asymptotic behaviour of $\SSys(L_m^{2n+1})$ in terms of both $m$ and $n$ is still an open question.

\section{Simplicial complexity}\label{sec:simplicialcomplexity}

In this section we introduce the definition of simplicial complexity and give some of its basic properties. We then compare this new complexity with the two other standard complexities, namely the $T$-invariant of Delzant \cite{Del96} and the $c$-complexity of Matveev \& Pervova \cite{MP01}. Next we show a central lower bound for the simplicial complexity in terms of the $1$-torsion of the group.

\subsection{Definition and examples}\label{sec:def}

Given a finite simplicial complex $P$ we denote by $s_k(P)$ the
number of its $k$-simplices.

\begin{definition}\label{def:simplicialcomplexity} Let $G$ be a group.
We define its {\it simplicial complexity} $\kappa(G)$ by the following formula:
$$
\kappa(G) := \mathop{\inf}\limits_{\pi_1(X)=G} s_2(X),
$$
the infimum being taken over all simplicial  $2$-complexes
$X$ with fundamental  group $G$.
A $2$-complex $X$ is then said  {\it minimal for $G$} if $\pi_1(X) = G$,  $s_2(X) = \kappa(G)$ and each vertex is incident to at least $2$ edges.
If $G$ is of free index zero the last condition is equivalent to the one that each vertex is incident to a face.
\end{definition}

It is important to note that a notion of complexity has already appeared in the context of topology and geometry of $3$-manifolds. This notion originally due to S.~Matveev relies on a special class of subpolyhedron called spine, and is in general very hard to compute. We refer the reader to the book \cite{Mat07} for a good introduction as well as several applications for this complexity. Also remark that there exist variations of this notion that are defined using pseudo-triangulations or triangulations of $3$-manifolds instead of spines (see the recent articles \cite{MPV09, JRT09, JRT11, JRT13}). The articles \cite{Cha15, Cha16} contain several comparison results between these various complexities.\\

Going back to our definition \ref{def:simplicialcomplexity}, we see that simplicial complexity satisfies the following properties:

\medskip

\noindent {\bf 1.} $\kappa(G) = 0$ if and only if $G$ is a free group.

\medskip

\noindent {\bf 2.} The free product of two groups $G_1$ and $G_2$ satisfies
\begin{equation}\label{ineq:kappa}
\kappa(G_1 \ast G_2) \leq \kappa(G_1) + \kappa(G_2).
\end{equation}
But simplicial complexity is not additive with respect to free product: if $\kappa(G_1)$ and $\kappa(G_2)$ are both positive, inequality (\ref{ineq:kappa}) can be strengthened by
$$
\kappa(G_1 \ast G_2) < \kappa(G_1) + \kappa(G_2).
$$
For this fix two $2$-complexes $X_1$ and $X_2$ which are minimal for $G_1$ and $G_2$ respectively and glue them together by identifying one $2$-simplex of $X_1$ with another $2$-simplex of $X_2$ (the choice of these two $2$-simplices being not relevant).

\medskip

\noindent {\bf 3.} For a simplicial complex $P$, its {\it
simplicial height} $h(P)$ is the total number of its simplices of any dimension. This invariant
was introduced in \cite{Gro96} and satisfies
$$
h(P) \geq \kappa(\pi_1(P)) .
$$

\medskip

\begin{example}\label{ex:table}
Even for groups whose structure is simple, the exact value of
$\kappa$ seems hard to compute. For small values up to $17$ the following table describes the situation, see \cite{Bul14}. Here $K_2$ denotes the fundamental group of the Klein bottle, while the annotation $(\ast)$ means that the corresponding minimal complex is unique and $?$ that there might be some others groups with the same simplicial complexity.

\medskip


\begin{center}
{\renewcommand{\arraystretch}{1.5}
\begin{tabular}{||p{1.5cm}||p{1.5cm}|p{1.5cm}|p{1.5cm}|p{1.5cm}|} \hline
\centering$\kappa(G)$&\centering {10}&\centering {14}&\centering{16}&{\centering {17}} \\\hline
\centering$G$ & \centering$\Z_2(\ast)$ & \centering$\Z\oplus\Z(\ast)$ &\centering$K_2$ & {\centering$\Z_3,\ ?$} \\\hline
\end{tabular}}
\end{center}
$\ $

\medskip
For instance there might be several groups with simplicial complexity $17$ as $17\le \kappa(\Z_2\ast\Z_2)\le 18$ according to \cite{Bul14}.
The unique minimal complex for $\Z_2$ is the quotient of the icosahedron by the central symmetry, see Figure \ref{fig-trigRP2}.

\begin{figure}[h]
\begin{center}
\includegraphics[width=0.5\linewidth]{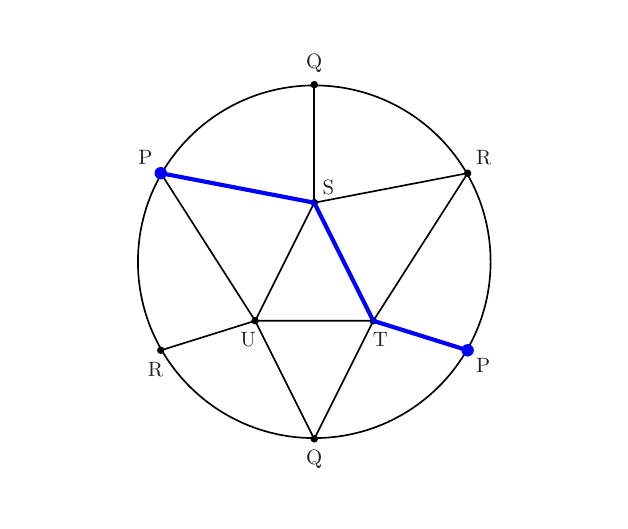}
\caption{Minimal complex for $\Z_2$.} \label{fig-trigRP2}
\end{center}
\end{figure}

For $\Z \oplus \Z$ the minimal complex is also unique and is given by the minimal triangulation of the $2$-torus whose fundamental domain is depicted in Figure \ref{fig-trigtore}.

\begin{figure}[h]
\begin{center}
\includegraphics[width=0.6\linewidth]{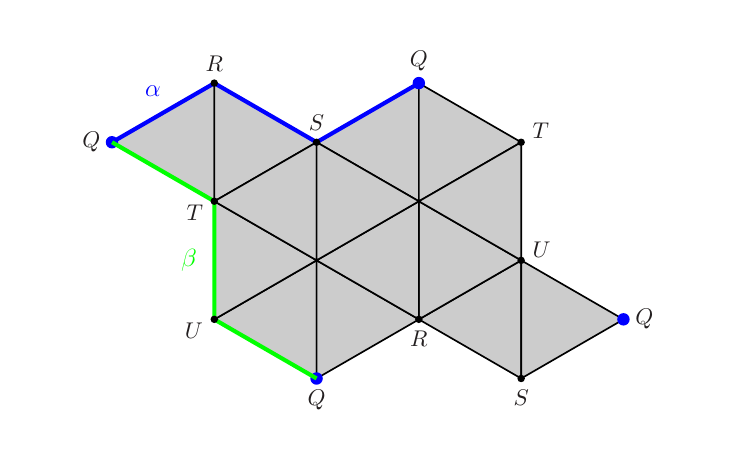}
\caption{Minimal complex for $\Z \oplus \Z$.} \label{fig-trigtore}
\end{center}
\end{figure}

These two minimal complexes will be used in the sequel.
\end{example}

\begin{example}\label{ex:surfaces}
For surface groups with large genus, the exact computation of
their complexity remains an open problem. We can nevertheless give
some bounds in terms of their genus.

Let $\pi_1 (S_l)$ be the fundamental group of an orientable surface of
genus $l \geq 1$. By elementary algebraic and combinatorial
considerations
\begin{equation} \label{eq:kappasurf1}
\kappa(\pi_1(S_l)) \geq {4\over3}l.
\end{equation}
Besides
\begin{equation} \label{eq:kappasurf2}
\kappa(\pi_1(S_l)) \leq 4(l-1) + 2\left\lceil{7 + \sqrt{1 + 48l}\over
2}\right\rceil
\end{equation}
by a result of Jungerman \& Ringel \cite{Ring}. Here $\lceil a\rceil$ denotes the integer part of $a+1$ if $a$ is
not an integer and $a$ for integers. Strictly speaking, the upper
bound is only available for $l \neq 2$. For $l = 2$ we have to
replace the upper bound by  $24$, see \cite{Ring}.
Observe that the upper bound (\ref{eq:kappasurf2}) is sharp for $\pi_1(S_1)=\Z \oplus
\Z$.
Because $\kappa(\Z \oplus \Z) = 14$  we can easily derive that the free abelian group $A_n$ of
rank $n$ satisfies
$$
{1 \over 2}n(n-1) \leq  \kappa(A_n) \leq 7n(n-1).
$$
Here the lower bound is given by the second Betti
number. For the upper bound consider the $n$-dimensional torus $\TT^n$ and the  $2$-dimensional skeleton of its standard cell decomposition. This $2$-skeleton consists of ${n(n-1) \over 2}$ tori of dimension $2$, all of which we endow with the minimal triangulation shown in Figure \ref{fig-trigtore}. These individual triangulations can be arranged in a global triangulation of the $2$-skeleton by choosing the blue and green lines in Figure \ref{fig-trigtore} to lie in the $1$-skeleton of the standard cell decomposition of  $\TT^n$. This gives a $2$-complex with fundamental group $A_n$ and a triangulation with exactly $7n(n-1)$ triangles. 

Remark that the precise computation for $\kappa(A_n)$ remains open.
\end{example}

Subadditive property (\ref{ineq:kappa}) implies that for any group $G$,
\begin{equation}\label{eq:*Z}
\kappa(G \ast \Z) \leq \kappa(G).
\end{equation}
As for systolic area the question to know whether or not this inequality is actually an equality is open and fundamental. Because of (\ref{eq:*Z}) we will consider in the sequel only groups with free index zero.

If $G$ is a free index zero group and $X$ is a minimal complex for $G$, it is straightforward to check that

\smallskip

        \noindent ($M_1$) any edge of $X$ is adjacent to at least two $2$-simplices,

\smallskip

        \noindent ($M_2$) any vertex of $X$ is adjacent to at least four $2$-simplices.

\smallskip

This properties of minimal complexes will be usefull in the sequel.

\subsection{Comparison with other complexities}\label{subsec:comparison}
There exist two other numerical invariants which measure the complexity of groups. First recall that given a presentation
$$
\mathcal P=\left\langle a_1,\ldots,a_n\mid r_1,\ldots,r_m\right\rangle
$$
of a group $G$, its length is the number
$$
\ell({\mathcal P})=\displaystyle \sum _{i=1}^m\left|r_i\right|
$$
where $\left|\, \cdot\,\right|$ denotes the word length associated to the system of generators $\{a_1,\ldots,a_n\}$. \\

The two other types of complexity for groups are the following:
\begin{itemize}

\item  the $c$-complexity introduced by Matveev \& Pervova \cite{MP01} and defined as the minimal length $c(G)$ of a finite presentation,

\item  the $T$-invariant introduced by Delzant \cite{Del96} and defined by
$$
T(G):=\min_{\mathcal P=\left\langle a_1,\ldots,a_n\mid r_1,\ldots,r_m\right\rangle} \sum_{i=1}^m
\max\left\{\left|{r_i}\right|-2,0\right\}.
$$
\end{itemize}

Any  group can be presented with relations of length either $2$ or $3$, such a representation being called {\it triangular}.
The $T$-invariant is nothing else than the minimal number of relations of length $3$ for a triangular representation.

One fundamental property of the $T$-invariant is its linearity with respect to free products (see \cite{Del96}):
$$
T(G_1 \ast G_2) = T(G_1) + T(G_2).
$$
But of course this invariant is not sensitive to torsion elements of order $2$, for instance
$$
T(\Z_2 \ast \dots \ast \Z_2) = 0.
$$
In \cite{PP08}(see also \cite{KS05}) it is shown that $T$-invariant and $c$-complexity satisfy the following relations:

\smallskip

\begin{itemize}
\item $T(G) \leq c(G)$,

\smallskip

\item $c(G) \leq 9 \, T(G)$ if $G$ does not admit free factor isomorphic to $\Z_2$,

\smallskip

\item $c(G) \leq 3 \, T(G)$ if $G$ has no torsion element of order $2$.
\end{itemize}

On the other hand $c$-complexity and simplicial complexity coincide up to some universal constants:

\begin{proposition}\label{prop:c-kappa.comparaison}
For any group  $G$ we have
\begin{equation}\label{eq:c-kappa.comparaison}
{1 \over 6} \, \big(\kappa(G)+h\big) \leq c(G) \leq 3 \, (\kappa(G) - 4),
\end{equation}
where $h$ is the minimal number of relations over finite presentations of $G$.
\end{proposition}
\begin{proof}
To prove the left-hand side we start with an arbitrary finite presentation
$$
\mathcal P= \langle a_1, \ldots , a_n | r_1, \ldots r_m \rangle
$$
of $G$. We can suppose this presentation contains no relation of length $1$ and that every relation of length $2$ is of the form $r_j = a_s^{\pm 2}$.
Associated to $\mathcal P$ we construct a simplical $2$-complex $X$ with $\pi_1(X)=G$ and $s_2(X)\leq 6 \, \ell(\mathcal P)-m$. Start with $Y =\vee_{i=1}^n S^1_i$ the wedge sum of $n$ circles whose base point is denoted by $P$, each circle corresponding to a generator of the presentation $\mathcal P$. For each $j=1,\ldots,m$ we glue a $2$-disk $D^2_j$ by identifying its boundary with the loop described by the relation $r_j$.  The topological space thus obtained is denoted by $X$  and can be triangulated as follows. First divide each circle into three edges, the base point $P$ corresponding to one of the vertices.
 If $|r_j| \geq 3$ we triangulate the disk $D^2_j$ according to the parity of $|r_j|$. If $|r_j|=2k$, we triangulate $D^2_j$ as in Figure \ref{fig-trigB2_pair} and get a contribution of $6|r_j|-2$ triangles (remark the blocks made of nine triangles). If $|r_j|=2k+1$, we triangulate $D^2_j$ as in Figure \ref{fig-trigB2_impair} and get a contribution of $6|r_j|-1$ triangles.

\begin{figure}[h]
\begin{center}
\includegraphics[width=0.6\linewidth]{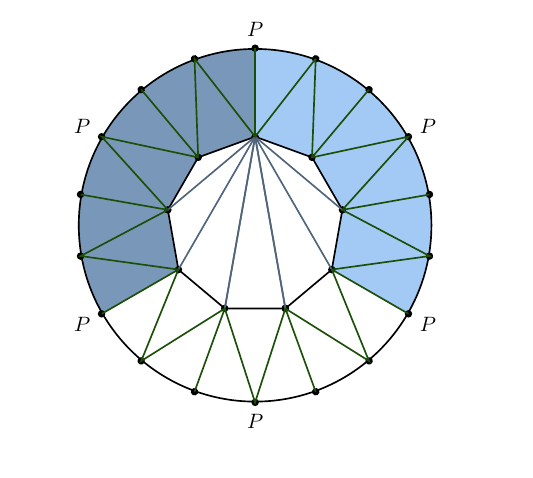}
\caption{Triangulation of $D^2_j$ for $|r_j|$ even ($|r_j|=6$).} \label{fig-trigB2_pair}
\end{center}
\end{figure}

\begin{figure}[h]
\begin{center}
\includegraphics[width=0.6\linewidth]{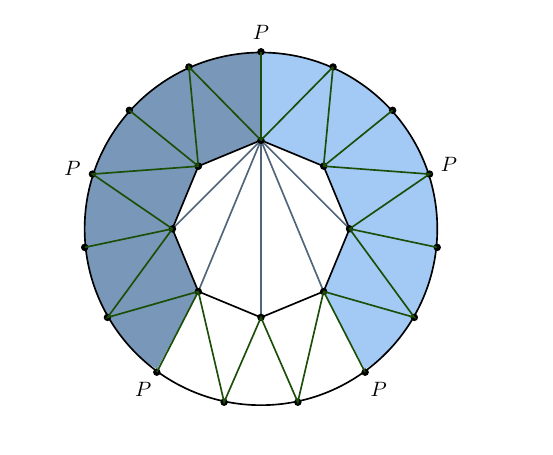}
\caption{Triangulation of $D^2_j$ for $|r_j|$ odd ($|r_j|=5$).} \label{fig-trigB2_impair}
\end{center}
\end{figure}

If $|r_j| = 2$ we triangulate the corresponding projective plane like in Figure \ref{fig-trigRP2} and get a contribution of $10= 5|r_j|$ simplices of dimension $2$.

Finally the triangulation of $X$ thus defined satisfies
$$
s_2(X)\leq 6\mathop{\sum}\limits_{j=1}^m |r_j|-m\leq 6 \, \ell(\mathcal P)-m.
$$
Because we start with any finite presentation $\mathcal P$ we conclude that $\kappa(G)+h \leq 6 \, c(G)$.\\

The right-hand side of the inequality is proved as follows. Consider a simplicial complex $X$ minimal for $G$ and having a minimal number of edges. Fix a maximal tree of the $1$-dimensional skeleton of $X$ with root $p$ such that each edge incident to $p$ belongs to the tree.
We first contract every $2$-simplex incident to $p$ and then the remaining part of the maximal tree into a point. This gives a finite cell complex $\tilde{X}$ of dimension $2$ whose $2$-cells are glued along at most $3$ cells of dimension $1$. Because $X$ is minimal, the root $p$ was incident to at least four $2$-simplices, so the number of $2$-cells of $\tilde{X}$ is  at most $s_2(X)-4$.
This topological space being homotopy equivalent to $X$ we get a presentation of $\pi_1 (\tilde{X})=G$ whose length is at most $3\, (s_2(X)-4)$.
\end{proof}
\begin{remark}
If $G$ has no torsion element of order $2$, the proof implies the following inequality:
$$
\kappa(G) \leq 5\, c(G) + T(G)\leq 13\, T(G).
$$
\end{remark}
\vskip7pt

\subsection{Lower bound in terms of $1$-torsion}\label{sub:section}

The simplicial complexity $\kappa(G)$ is quite sensitive to the
number of torsion elements in $H_1(G, \Z)$. The next proposition
will be used several times in the sequel.

\begin{proposition}\label{prop:simplexe2} Let $X$ be a simplicial complex of dimension $2$. Then
$$
s_2(X) \geq 2 \, \log_3 \, |\mbox{Tors}\,H_1(X, \Z)|.
$$
In particular, every group $G$ satisfies the
inequality
$$
\kappa(G) \geq 2\, \log_3 \, |\mbox{Tors}\,H_1(G, \Z)|.
$$
\end{proposition}

\begin{proof} Consider the complex of simplicial cochains
\begin{equation}\label{eq:cochanes}
C^1(X, \Z) \mathop{\longrightarrow}\limits^{d^1} C^2(X, \Z)
\mathop{\longrightarrow}\limits^{d^2} 0.
\end{equation}
The universal coefficient theorem implies a duality between
homology torsion and cohomology torsion, and we have  (see
\cite[Corollary 3.3]{Hat02} for instance)
\begin{equation}\label{eq:dual}
\mbox{Tors}\, H_1(X, \Z) \simeq \mbox{Tors}\, H^2(X, \Z).
\end{equation}
This implies that $|\mbox{Tors}\,H_1(X, \Z)|=
|\mbox{Tors}\big{(}C^2(X, \Z)/\mbox{Im} \, d^1\big{)}|$.

We endow $C^i(X, \Z)$ with the basis dual to the simplicial basis
of $C_i(X, \Z)$.  Let $D$ denote the matrix of $d^1$ with respect
to these bases. The matrix $D$ has  $s_2(X)$ rows, and each row
has exactly three non zero elements whose value is either $1$ or
$-1$. It follows that every row vector of $D$ has Euclidean length
$\sqrt 3$. If we interpret the determinant of a square matrix $V$
of order $k$ as the volume of the parallelotope generated by its
row vectors, we see that for any square submatrix  $V$ of $D$
of order $k$
$$
|\mbox{det}V| \leq ({\sqrt 3})^k.
$$
Now if we denote by $t(D)$ the greatest common divisor of all minors of $D$ of order $\text{rank}(D)$, we deduce that
$$
t(D) \leq (\sqrt{3})^{\text{rank}(D)} \leq (\sqrt{3})^{s_2(X)}.
$$
To conclude just remember that 
$$
t(D)=|\mbox{Tors}\left(\mbox{Coker} \, d^1\right)|=|\mbox{Tors}\left(C^2(X, \Z)/\mbox{Im} \, d^1\right)|.
$$
\end{proof}

\begin{remark}
It is proved in \cite{PP08} that for any  group $G$
$$
c(G)\geq3\log_3|\text{Tors}H_1(G, \Z)|.
$$
This combined with the right-hand side of inequality (\ref{eq:c-kappa.comparaison}) leads to a weaker estimate of  $\kappa(G)$ than the one obtained in Proposition \ref{prop:simplexe2}.
\end{remark}

\subsection{Stabilization for free products}

Given a group $G$ we denote by
$$
G_{(n)} = \underset{n}{\underbrace{G \ast \ldots \ast G}}
$$
the free product of $n$ copies of $G$. As the function $\kappa(G_{(n)})$ is sublinear in $n$ according to inequality (\ref{ineq:kappa}), we can define the {\it stable simplicial complexity} by
$$
\kappa_{\infty}(G) := \mathop{\lim}\limits_{n \rightarrow \infty}{\kappa\big{(}G_{(n)}\big{)}\over n}.
$$

In this section we show that $\kappa_{\infty}(G)>0$ for any group $G$ which is not free. That is, albeit simplicial complexity is not additive, its asymptotic behaviour for free products of the same group is essentially linear. The analogous question for systolic area is completely open.

\begin{proposition}\label{prop:kappa(G_(n))}
Any group $G$ which is not free satisfies
$$
\kappa(G) - 1 \geq \kappa_{\infty}(G) \geq \left\{
\begin{matrix}
2 \log_3 2 &  \text{if} \, \, G \, \, \text{decomposes as} \, \, G = G'
\ast \Z_2,  \\
& \\
\frac{T(G)}{3} & \text{if} \, \, G  \, \,
\text{does not admit such a decomposition}.
\end{matrix}
\right.
$$
\end{proposition}

\begin{proof}
We start with a minimal complex $X$ for $G$. Choose a triangle $\Delta \subset X$ and consider the complex
$$
Y_n = \underset{n}{\underbrace{X \mathop{\cup}\limits_{\Delta} \dots \mathop{\cup}\limits_{\Delta} X}}.
$$
We get
$$
\kappa(G_{(n)}) \leq s_2(Y_n) = n(s_2(X)-1) + 1 =n(\kappa(G) - 1) + 1
$$
which implies the upper bound.
Now if $G = G' \ast \Z_2$ then
$$
|\text{Tors}H_1(G_{(n)}, \Z)| \geq |\text{Tors}H_1((\Z_2)_{(n)},
\Z)| = 2^n.
$$
Using Proposition \ref{prop:simplexe2} we derive
$$
\kappa_{\infty}(G) \geq 2\log_3 2.
$$
If  $G$ does not decompose as $G' \ast \Z_2$, because $T(G) \leq c(G) \leq 3\, (\kappa(G)-4)$ we get the result as $T$-invariant is additive for free products.
\end{proof}

\section{Simplicial complexity and systolic area}\label{sec:complexityvssystolicarea}

This section is completely devoted to the proof of Theorem \ref{th:syscomplexity} given in the introduction.\\

For the left-hand side inequality we consider a minimal simplicial
complex $X$ of dimension $2$ with fundamental group $G$. Endow $X$
with the metric $h$ such that any edge is of length $2\pi \over 3$
and any face is the round hemisphere of radius $1$. Because $X$ is
minimal, $s_2(X) = \kappa(G)$ and so
$$
\mbox{area}(X, h) = 2\pi\kappa(G).
$$
The definition of  the metric $h$ implies that any systolic geodesic can be
homotoped to the $1$-skeleton without increasing its length. Such
a curve passes through at least three edges and thus $\sys(X, h)
\geq 2\pi$. This implies that
$$
\sigma(G) \leq \sigma(X, h) \leq {\kappa(G) \over 2\pi}.
$$

\bigskip

To prove the right-hand side inequality we use several ideas introduced in
\cite[5.3.B]{Gro83}. Set  
\begin{equation}\label{alpha}
\alpha=25\exp\left(\sqrt{\ln (50^2 \cdot \sigma(G))}\right).
\end{equation}
By \cite[Theorem 3.5]{RS08},
$$
\sigma(G) \geq {1 \over 4}
$$
as $G$ has free index zero.
The real number $\alpha$ is thus well defined and satisfies $\alpha>25$. Fix some
positive $\varepsilon<{1 \over 25}$ small enough such that
\begin{equation}\label{epsilon}
\log_5 \frac{1}{25\varepsilon} \cdot \ln \frac{\alpha}{25} \geq
\ln \, (50^2 \cdot \kappa(G)).
\end{equation}

\medskip

According to \cite[Theorem 3.5 and Lemma 4.2]{RS08} there exists a simplicial
complex $P$ of dimension $2$ endowed with a metric $g$ such that
\begin{itemize}
\item $\pi_1(P) = G$ ;
\item $\sys(P,g)=1$ ;
\item $\sigma(P, g)=\mbox{area}(P,g) < \sigma(G) + \varepsilon$ ;
\item any ball $B(p,R) \subset P$ of radius $R \in [\varepsilon,{1 \over 2}]$ centered at any point $p \in P$ satisfies the inequality
$$
|B(p,R)| \geq {1 \over 4} R^2.
$$
Here $|B(p,R)|$ denote the area of the metric ball $B(p,R)$ for the metric $g$. 
\end{itemize}

\medskip

Suppose first that there exists some point $p$ in $P$ such that 
$$
|B(p,5R)| \geq  \alpha |B(p,R)| 
$$
for all $R \in [\varepsilon,{1 \over 25}]$ (the reasons for considering those special values will become clear in the proof of Lemma \ref{lem:nerf} and the final calculation). Then, if $r$ denotes the unique integer such that 
$$
{1\over 5^{r+2}} \leq \varepsilon < {1 \over 5^{r+1}},
$$
we can compute that
$$
|B(p,1/5)| \geq  \alpha^r |B(p,1/5^{r+1})| \geq {1\over 4}
\alpha^r \varepsilon^2 
\geq {1\over 50^2} \left(\frac{\alpha}{25}\right)^r 
\geq {1\over 50^2} \left(\frac{\alpha}{25}\right)^{\log_5{1 \over 25 \varepsilon}}.
$$
Thus
$$
\sigma(P,g)=\mbox{area}(P,g)\geq |B(p,1/5)| \geq\kappa(G)
$$
according to the inequality (\ref{epsilon}) and the result is proved in this case. \\

Now suppose that for all $p$ in $P$ 
$$
|B(p,5R)| <  \alpha |B(p,R)| 
$$
for some $R \in [\varepsilon,{1 \over 25}]$ and let $R_p$ denote the supremum of such radii. In particular
$$
| B(p,5R_p)| \leq \alpha \cdot |B(p,R_p)|,
$$
and  for all $R \in ]R_p,{1\over 25}]$
$$
|B(p,5R)|\geq\alpha \cdot |B(p,R)|.
$$
 Remark that if $R_p<1/25$ then 
$$
| B(p,5R_p)| = \alpha \cdot |B(p,R_p)|.
$$

Such a ball $B(p,R_p)$ is said to be {\it $\alpha$-admissible} according to Gromov's terminology (see \cite[Theorem 5.3.B]{Gro83}) and has area bounded from below as follows.

\begin{lemma}
$$
|B(p,R_p)|\geq A:={1\over 50^2} \left({1 \over 25}\right)^\frac{\ln {50^2 \, (\sigma(G)+\varepsilon)}}{\ln (\alpha / 25)}.
$$
\end{lemma}

\begin{proof}[Proof of the lemma]
Let $r$ be the unique integer such that 
$$
{1 \over 5^{r+2}} < R_p \leq {1 \over 5^{r+1}}.
$$
We have
$$
\sigma(P,g)=\mbox{area}(P,g) \geq |B(p,{1 / 5})| \geq \alpha^r |B(p,R_p)| \geq
{1\over 4} \alpha^r R_p^2 \geq {1\over 50^2} \left({\alpha \over
25}\right)^r,
$$
and so
$$
r\leq \frac{\ln {50^2 \, (\sigma(G)+\varepsilon)}}{\ln {\alpha
\over 25}}.
$$
This implies
$$
|B(p,R_p)|\geq {1\over 4} R_p^2 \geq {1\over 50^2} \left({1 \over
25}\right)^r \geq {1\over 50^2} \left({1 \over 25}\right)^\frac{\ln {50^2 \, (\sigma(G)+\varepsilon)}}{\ln (\alpha / 25)}.
$$
\end{proof}

\medskip

Using this familly of $\alpha$-admissible balls, we now construct a cover of $(P,g)$ whose nerve has $G$ as fundamental group and whose number of $2$-simplices is bounded from above by the systolic area of $(P,g)$. This will lead to the desired upper bound on the simplicial complexity of $G$ in terms of its systolic area. \\

If $B:=B(p,R)$ denotes some metric ball centered at $p$ and of radius $R$, we denote by $nB$ the concentric ball
$B(p,nR)$ for any positive integer $n$.
First choose  an $\alpha$-admissible ball $B_1:=B(p_1,R_1)$ with
$R_1:=\max\{R_p \mid p \in P\}$. At each step $i\geq 2$, we
construct $B_i$ using the data of $\{B_{j}\}_{j<i}$ as follows.
Let $R_i$ be the maximal radius of an  $\alpha$-admissible ball
centered at a point in the complement of the union of the balls
$\{2B_j\}_{j<i}$ and let $B_i:=B(p_i,R_i)$ be such an
$\alpha$-admissible ball. By construction, $B_i$ is  disjoint from
the other balls  $B_{j}$ as  $R_i \leq R_j$. The process ends in a
finite $N$ number of steps when the balls
$\{2B_i\}_{i=1}^N$ cover $P$, as $R_i\geq \varepsilon$ for every $i=1,\ldots,N$. \\

Consider the corresponding nerve ${\mathcal N}$ of this cover. In
general, if $X$ is a  paracompact topological space and $\U$ a
locally finite cover of $X$, there exists a canonical map $\Phi$
from $X$ to the nerve ${\mathcal N}(\U)$ of the cover $\U$ defined
as follows. If $\{\phi_V\}_{V \in \U}$ denotes a partition of
unity associated with $\U$,
\begin{eqnarray*}
\Phi : X & \to & {\mathcal N}(\U)\\
x & \mapsto & \mathop{\sum}\limits_{V \in \U}\phi_V(x)V.\\
\end{eqnarray*}
This map is uniquely defined up to homotopy. In our case,
$\U=\{2B_i \}_{i=1}^N$ and $\Phi$ associates the center of such a
ball to the corresponding vertex of ${\mathcal N}$.

\begin{lemma} \label{lem:nerf}
The map $\Phi : P \to {\mathcal N}$ induces an isomorphism of
fundamental groups.
\end{lemma}

\begin{proof}[Proof of the lemma] Denote by  ${\mathcal N}^{(k)}$
the $k$-skeleton of ${\mathcal N}$. We will construct a map
$\Psi : {\mathcal N}^{(2)} \to X$ such that the induced map
$$
\Psi_\sharp : \pi_1({\mathcal N})\simeq \pi_1({\mathcal N}^{(2)})
\to \pi_1(P)
$$
is the inverse of $\Phi_\sharp : \pi_1(P) \to \pi_1({\mathcal
N})$.

\medskip

Let define $R_0:=1/25$ and recall that for all $i=1,\ldots,N$ we have $R_i \leq R_0$.
We have denoted the set of centers of the balls of the covering
$\U$ by $\{p_i\}_{i=1}^N$ and set $v_i=\Phi(p_i)$. We first define
$\Psi$ on ${\mathcal N}^{(0)}$ by
$$
\Psi(v_i) =p_i.
$$

If two vertices $v_i$ and $v_j$ are connected by an edge $[v_i,
v_j]$, we join $p_i$ and $p_j$ in $P$ by any minimizing geodesic
denoted by $\gamma_{i,j}$. The map $\Psi$ is then defined on the
edge $[v_i,v_j]$ to the arc $\gamma_{i,j}$ in the obvious way
$$
\Psi : [v_i,v_j] \longrightarrow \gamma_{i,j}.
$$
This defines  $\Psi$ on the $1$-skeleton ${\mathcal N}^{(1)}$.
Remark that the curve $\gamma_{i,j}$ has length less than ${4 \cdot R_0}$ ($v_i$ and $v_j$
are connected by an edge if and only if $2B_i \cap 2B_j\neq
\emptyset$).

Next we consider a $2$-simplex $\tau=[v_i, v_j, v_k]$ of
${\mathcal N}$. The concatenation $\gamma_{i,j} \star
\gamma_{j,k}\star \gamma_{k,i}$ is a closed curve of $P$ of length
less than $12 \cdot R_0<1$. So it is contractible and any
contraction of this curve into a point gives rise to an extension
of the map $\Psi$ to $\tau$. We get this way a map
$$
\Psi : {\mathcal N}^{(2)} \to P.
$$
Observe that the restriction of $\Psi$ to ${\mathcal N}^{(1)}$ is
unique up to homotopy.

\medskip

By construction, $\Phi(p_i) = v_i$ for any $i=1,\ldots,N$, and if
$[v_i,v_j]$ denotes an edge of ${\mathcal N}$ and $p$ belongs to
the corresponding geodesic $\gamma_{ij}$, $\Phi(p) \in
\mbox{St}([v_i, v_j])$ where $\mbox{St}([v_i, v_j])$  denotes the
star of $[v_i, v_j]$. This implies that $\Phi \circ \Psi: [v_i,
v_j] \longrightarrow {\mathcal N}$ is homotopically equivalent to
the identity relatively to $\{v_i, v_j\}$. So $\Phi \circ \Psi:
{\mathcal N}^{(1)} \longrightarrow {\mathcal N}$ is homotopically
equivalent to the identity relatively to ${\mathcal N}^{(0)}$.
From this, we get that the induced morphism $\Phi_\sharp \circ
\Psi_\sharp : \pi_1({\mathcal N})\to \pi_1({\mathcal N})$ is the
identity and so $\Psi_\sharp: \pi_1({\mathcal N})\to \pi_1(P)$ is
injective.

\medskip

It remains to prove that $\Psi_\sharp$ is onto. Consider a
geodesic loop $\delta$ based at the center $p_1$ of the ball $B_1$
and whose length is minimal in its own homotopy class. We complete
$p_1$ into a finite family $\{p_{i_j}\}_{j \in \Z_n}$ of points of
$P$ such that
\begin{itemize}
\item each $p_{i_j}$ is the center of some ball $B_{i_j}$ of $\U$ ;
\item the family $\{2B_{i_j}\}_{j \in \Z_n}$ covers $\delta$ ;
\item $2B_{i_j} \cap 2B_{{i_{j+1}}}\neq \emptyset$.
\end{itemize}
For each $j \in \Z_n$, fix any point $x_j \in 2B_{i_j} \cap
\delta$ and denote by $\delta_j $ the part of the loop $\delta$
joining $x_j$ and $x_{j+1}$ and contained in $2B_{i_j} \cup
2B_{{i_{j+1}}}$. By construction, the curve $\delta_j$ has length less than $8 \cdot
R_0$. Fix a minimizing geodesic $\beta_j$ joining $p_{i_j}$ and
$x_j$. The concatenation
$$
\gamma_{i_j,i_{j+1}}\star \beta_{j+1} \star (\delta_j)^{-1}  \star
(\beta_j)^{-1}
$$
is a closed curve of length less than $24 \cdot R_0<1$ and thus
contractible. So $\delta$ is homotopic to  $\gamma_{1,2}\star
\gamma_{2,3}\star ... \star\gamma_{n,1}$ with based point $p_1$
fixed. This proves the surjectivity of $\Psi_\sharp$ and completes
the proof.
\end{proof}

\medskip

So $\pi_1({\mathcal N}) \simeq G$, and we deduce the lower bound
$$
s_2({\mathcal N}) \geq \kappa(G).
$$
We now estimate the number $s_2({\mathcal N})$ by the systolic area of $(P,g)$.\\
 
Recall that $B_1,\ldots,B_N$ denote pairwise disjoint $\alpha$-admissible balls of radii $R_1,\ldots,R_N$ satisfying $R_i\geq R_{i+1}$. For $i=1,\ldots,N-1$ denote by $I_i$ the set of $j>i$ such that $2B_i\cap 2B_j \neq \emptyset$, and observe that if $j \in I_i$ then $B_j \subset 5B_i$. Then

\begin{eqnarray*}
\sigma(P,g)=\text{area}(P,g)&\geq & \sum_{i=1}^N |B_i|\geq \alpha^{-1} \sum_{i=1}^N |5B_i|\\
&\geq &\alpha^{-1} \sum_{i=1}^N \, \, \sum_{j\in I_i} |B_j|\\
&\geq &\alpha^{-2} \sum_{i=1}^N \, \, \sum_{j\in I_i} |5B_j|\\
&\geq &\alpha^{-2} \sum_{i=1}^N \, \, \sum_{j\in I_i} \, \, \sum_{k\in I_j} |B_k|\\
&\geq &\alpha^{-2} \sum_{i=1}^N \, \, \sum_{j\in I_i} \, \, \sum_{k\in I_j} A.\\
\end{eqnarray*}

But the set of $i<j<k$ such that $2B_i\cap 2B_j\neq \emptyset$ and $2B_j\cap 2B_k\neq \emptyset$ includes the set of $i<j<k$ such that $2B_i\cap 2B_j\cap 2B_k \neq \emptyset$, thus its cardinal is at least equal to $s_2(\mathcal{N})$, and we deduce

$$
s_2(\mathcal{N})\leq \alpha^2 \cdot {\sigma(P,g) \over A}.
$$
So
$$
\kappa(G) \leq \alpha^2Ê\cdot \frac{ (\sigma(G)+\varepsilon)}{{1\over 50^2} \left({1 \over 25}\right)^\frac{\ln {50^2 \, (\sigma(G)+\varepsilon)}}{\ln (\alpha / 25)}}\leq \alpha^2Ê\cdot {50^2} \cdot 25^\frac{\ln {50^2 \, (\sigma(G)+\varepsilon)}}{\ln (\alpha / 25)}\cdot (\sigma(G)+\varepsilon).
$$
From the equality
$$
\ln {\alpha \over 25}=\sqrt{\ln (50^2 \cdot \sigma(G))},
$$
we then compute that
$$
\kappa(G) \leq 25^2 \cdot  e^{2\sqrt{\ln (50^2 \cdot \sigma(G))}} \cdot  e^{\ln 25 \cdot \frac{\ln {50^2 \, (\sigma(G)+\varepsilon)}}{\sqrt{\ln (50^2 \cdot \sigma(G))}}}\cdot  50^2(\sigma(G)+\varepsilon).
$$
This finally implies the result by letting $\varepsilon \to 0$.

\section{Finitude results for simplicial complexity}\label{sec:finitude}

We focus in this section on finitude problems for the invariant $\kappa$: how to estimate the (obviously finite) number of isomorphism classes of groups whose simplicial complexity is at most $T$ ? Recall that given a positive integer $T$ the set $\G_\kappa(T)$ is defined as the isomorphism classes of groups $G$ with free index zero such that $\kappa(G)\leq T$ while $\mathcal{A}_{\kappa}(T)$ denotes the subset corresponding to finite abelian groups.

\subsection{Upper bound for $|\mathcal{G}_{\kappa}(T)|$}
We start with the proof of the upper bound contained in Theorem \ref{th:borne.sup}, namely
$$
|\mathcal{G}_{\kappa}(T)| \leq 2^{6T\log_2T}.
$$
Start with any simplicial complex $X$ of dimension $2$ minimal for $G$. Because $G$ is of free index zero and the simplicial complex $X$ is minimal recall that

($M_1$) : any edge of $X$ is adjacent to at least two $2$-simplices,

($M_2$) : any vertex of $X$ is adjacent at least to four $2$-simplices.

\noindent We derive   from ($M_1$) and  from ($M_2$) the following  upper bounds on the number of $0$- and $1$-simplices:
$$
s_0(X) \leq \frac{3T}{4} \, \, \text{and} \, \, s_1(X) \leq \frac{3T}{2}.
$$

Consider the barycentric subdivision $\text{sd}(X)$ of $X$ and color the vertices as follows.
\begin{quote}
        \begin{itemize}
        \item The original vertices of $X$ are colored in black,
        \item barycenters of edges of $X$ are colored in green,
        \item barycenters of faces of $X$ are colored in red.
        \end{itemize}
\end{quote}

Consider the $1$-skeleton of $\text{sd}(X)$ and erase the edges joining red and black vertices (see Figure \ref{fig-subdivision}).
We denote by $b$, $g$ and $r$ the number of respectively black, green and red vertices.
The $1$-dimensional simplical complex thus obtained is a $3$-colored graph which satisfies the following properties:

\begin{figure}[h]
\begin{center}
\includegraphics[width=0.4\linewidth]{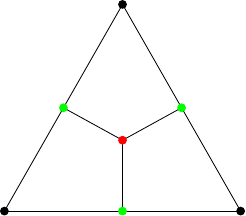}
\caption{The $3$-colored graph associated to a face.} \label{fig-subdivision}
\end{center}
\end{figure}

\smallskip

$(P_1)$ : $b  \leq {3T \over 4}$, $g  \leq {3T \over 2}$ and  $r \leq T$,

\smallskip

$(P_2)$ : any green vertex is adjacent to exactly two black vertices,

\smallskip

$(P_3)$ : any red vertex is adjacent to exactly three green vertices,

\smallskip

$(P_4)$ : no pair of red and black vertices are adjacent.

\smallskip

It is straightforward to check that to any minimal complex of a group with free index zero and simplicial complexity at most $T$ corresponds a unique $3$-colored graph satisfying properties $(P_1)$ to $(P_4)$. Observe that there exist $3$-colored graphs satisfying properties $(P_1)$ to $(P_4)$ which does not correspond to any minimal complex, and even does not correspond to any simplicial complex.

The number of $3$-colored graphs satisfying properties $(P_1)$ to $(P_4)$ can be estimated using their incidence matrix which has the following form
$$
\begin{pmatrix} 0 & A^t & 0 \\
A & 0 & B^t \\
0 & B & 0
\end{pmatrix}
$$
where $A$ is a $g \times b$ matrix with each row containing exactly two non zero coefficients equal to $1$, and $B$ is a $r \times g$ matrix for which each row contains exactly three non zero coefficient equal to $1$. Thus the number of such matrices $A$ is at most
$$
\left({b\, (b-1)\over 2}\right)^g
$$
and the number of such matrices $B$ is at most
$$
\left({g\,(g-1)\,(g-2)\over 6}\right)^r.
$$

From this we compute that the number of $3$-colored graphs satisfying properties $(P_1)$ to $(P_4)$ is at most
$$
{9T^3 \over 8} \cdot \left({1 \over 2} \cdot {3T \over 4} \cdot\left({3T \over 4} -
1\right)\right)^{3T \over 2} \cdot \left({1\over 6} \cdot {3T \over
2}\cdot \left({3T \over 2} - 1\right) \cdot \left({3T \over 2} - 2\right)\right)^T
$$
which is less than  $T^{6T}$ for  $T\geq 2$.
This concludes the proof.

\subsection{Simplicial complexity for finite abelian groups.}

In this subsection we shall see that the subset   $\mathcal{A}_{\kappa}(T)$ of {\it finite abelian} groups with simplicial complexity bounded by $T$ is already large.

\smallskip

Recall that a finite abelian group $G$ decomposes in a direct sum
$$
G = \Z_{n_1} \oplus \Z_{n_2} \oplus \ldots \oplus \Z_{n_s}
$$
where $n_1|n_2|\dots |n_s$. These numbers are called the {\it invariant factors} of $G$ and are uniquely defined by the group.
This decomposition can be used to estimate the simplicial complexity of $G$ as follows.

\begin{theorem}\label{teo:cyclique}
Any finite abelian group $G$ satisfies the double inequality
$$
2\log_3|G| \leq \kappa(G) \leq 14\log_2|G| + 7s^2 - 4s,
$$
where $s$ is the number of invariant factors of $G$.
\end{theorem}

\begin{remark}
In particular
$$
\kappa(G) \leq 7(\log_2|G|)^2 + 10\log_2|G|
$$
for any finite abelian group. The order of this upper bound is asymptotically sharp as shown by the following example. If
$$
G = \mathop{\oplus}\limits_{i=1}^s (\Z_2)_i,
$$
we see that $|G| = 2^s$ and $H_2(G,\Z_2)$ is a $\Z_2$-vector space of dimension $s(s+1) \over 2$. So
$$
\kappa(G) \geq {s(s+1) \over 2} \geq {1 \over 2} (\log_2 |G|)^2.
$$
\end{remark}

\bigskip

From Theorem \ref{teo:cyclique} we directly derive that for any $m\geq 2$
$$
2 \log_3 m \leq \kappa(\Z_m) \leq 14\log_2 m+3.
$$
In particular we get the lower bound announced in Theorem \ref{th:borne.sup}:

\begin{corollary}\label{cor:number.inf}
For any positive $T$
$$
|\mathcal{A}_{\kappa}(T)| \geq \left[2^{{T - 3} \over 14}\right]
$$
where $[x]$ denotes the integral part of the number $x$.
\end{corollary}

The rest of the subsection is devoted to the proof of Theorem \ref{teo:cyclique}.

\medskip

\begin{proof}[Proof of Theorem  \ref{teo:cyclique}]
From Proposition \ref{prop:simplexe2}
$$
\kappa(G) \geq 2\log_3|\text{tors}H_1(G, \Z)|
$$
which gives the left-hand side inequality of Theorem \ref{teo:cyclique} as $|\text{tors}H_1(G, \Z)|=|G|$ for a finite abelian group.

\medskip

The proof of the right-hand side inequality of Theorem \ref{teo:cyclique} relies on the following estimate for the simplicial complexity of $\Z_m$:

\begin{lemma}\label{lemme:1}
$$
\kappa(\Z_m) \leq 14 \log_2m + 3.
$$
\end{lemma}

\begin{proof}[Proof of the Lemma]
Start with a Moebius strip denoted by $\mathcal M$ and fix (see Figure \ref{fig-Moebius0})
\begin{itemize}
\item a point $P$ of its boundary $\partial \mathcal M$,
\item a simple loop $\gamma$ based at $P$ such that $\gamma \setminus \{P\}$ lies in the interior of $\mathcal M$.
\end{itemize}
In particular
$$
\{\partial {\mathcal M}\} = 2\{\gamma\} \in \pi_1({\mathcal M})
$$
and the class of $\gamma$ generates the fundamental group of $\mathcal M$.

\begin{figure}[h]
\begin{center}
\includegraphics[width=0.45\linewidth]{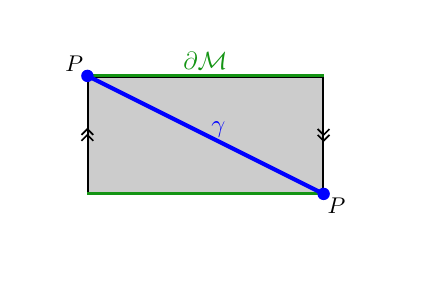}
\caption{The Moebius strip.} \label{fig-Moebius0}
\end{center}
\end{figure}

Let $\{\mathcal{M}_k\}_{k \in \N}$ be  an infinite number of copies of the Moebius strip $\mathcal{M}$ where the point corresponding to $P$ and the loop corresponding to $\gamma$ in the copy $\mathcal{M}_n$ are denoted respectively by $P_n$ and $\gamma_n$. 
We define a {\it Moebius telescope $\mathcal{T}_n$ of height $n$} as follows. 
Start with $\mathcal{T}_1=\mathcal{M}_0$  and then define  by induction
$$
\mathcal{T}_{n+1}=\mathcal{T}_{n} \mathop{\cup}\limits_{\varphi_{n}} \mathcal{M}_{n}
$$
where $\varphi_{n}$ is a homeomorphism between $\gamma_{n}$ and $\partial \mathcal{M}_{n-1}$ that sends $P_{n}$ to $P_{n-1}$. All gluing homeomorphisms will be chosen to be piecewise linear in the sequel.

 Observe that
 \begin{itemize}
 \item $\mathcal{T}_1 \subset \ldots \subset \mathcal{T}_{n-1} \subset \mathcal{T}_n$,
 \item all points $P_i \in \mathcal{M}_i$ for $i=0,\ldots,n-1$ glues onto a same point denoted
 $P$,
 \item $\gamma_0$ is a deformation retract of $\mathcal{T}_n$, and thus $\pi_1(\mathcal{T}_n) =\Z$,
 \item $\{\gamma_{i}\} = 2^{i}\{\gamma_0\}$ for $i=0,\ldots,n-1$.
 \end{itemize}

Fix a positive integer $m$ and define $n$ to be the smallest integer such that $m < 2^{n+1}$.
The dyadic decomposition of $m$ writes $m = 2^{n_1}+\ldots+2^{n_s}$ for some integers $0 \leq n_1 < \ldots < n_s = n$. In particular we have $s\leq n+1$. Let $\xi(m) = \gamma_{n_1}\star \gamma_{n_2}\star \ldots \star\gamma_{n_{s-1}}\star\partial \mathcal M_{n-1} \in \mathcal{T}_{n}$ be the loop based at $P$ ($\star$ denoting the concatenation operation for based loops). Consider the $2$-cell complex
$$
X_m = \mathcal{T}_{n} \mathop{\cup}\limits_{\xi(m)}D^2
$$
where the boundary of the $2$-disk $D^2$ is glued along the curve $\xi(m)$. Because $\{\xi(m)\} = m \{\gamma_0\}$ we get $\pi_1(X_m) = \Z_m$.

We shall construct an economic triangulation for $\mathcal{T}_{n}$, and then for $X_m$. Start with the minimal triangulation of $\R P^2$ which consists of $10$ triangles (see Figure
\ref{fig-trigRP2} and compare with Figure \ref{fig-Moebius1}). We fix a vertex $P$ and choose a simplicial loop $\gamma$ based at $P$ which generates $\pi_1(\R P^2)$ such as in Figure \ref{fig-Moebius1}. By deleting the interior of a triangle one of whose vertices is $P$ but which is not adjacent to any edge of the loop $\gamma$, we obtain a triangulation of $\mathcal{M}$ with $9$ triangles, see Figure \ref{fig-Moebius1}.

\begin{figure}[h]
\begin{center}
\includegraphics[width=0.8\linewidth]{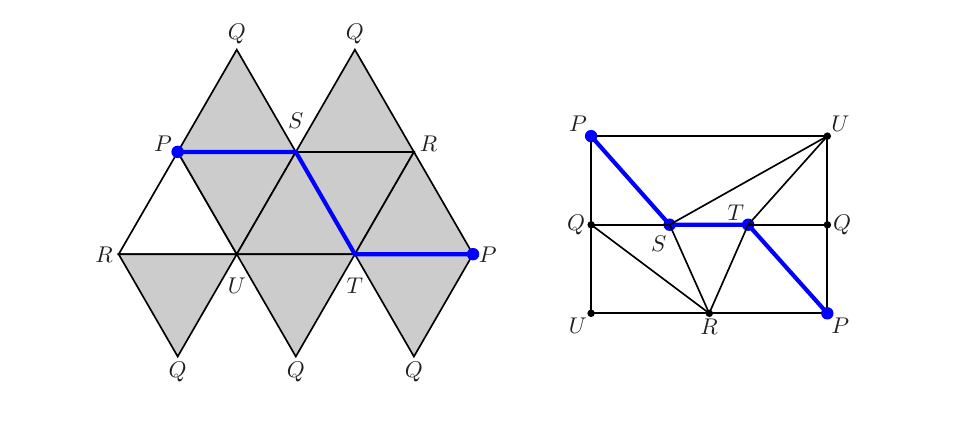}
\caption{A special triangulation of the Moebius strip.} \label{fig-Moebius1}
\end{center}
\end{figure}

If each Moebius strip of the Moebius telescope of height $n$ is endowed with this triangulation, we get a triangulation of $\mathcal{T}_{n}$ with  $9n$ triangles. Observe that the loop $\xi(m)$ consists of exactly $3s$ edges.
Now triangulate the $2$-disk $D^2$ by using at most $5s-2$ triangles as in Figure \ref{fig-trigB2} (compare with the proof of Proposition \ref{prop:c-kappa.comparaison}).

\begin{figure}[h]
\begin{center}
\includegraphics[width=0.5\linewidth]{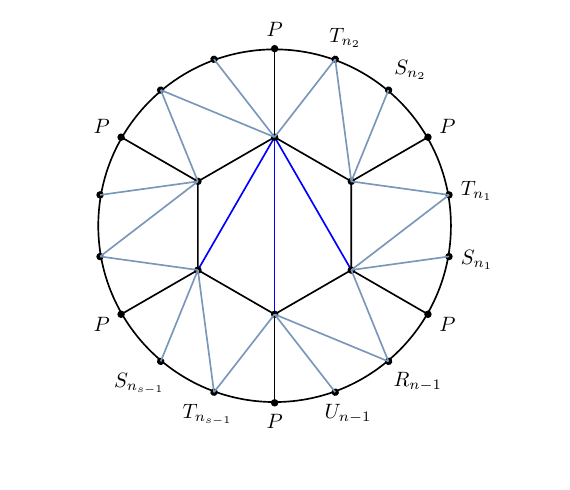}
\caption{Triangulation of $D^2$ for $s=6$.} \label{fig-trigB2}
\end{center}
\end{figure}

The triangulation of $X_m$ thus obtained satisfies
$$
s_2(X_m)\leq 9n + 5s-2 \leq 14n+3 \leq 14 \log_2m +3
$$
which concludes the proof of the lemma.
\end{proof}

\begin{remark}
The group $\Z_m$ can be realized as the fundamental group of a $2$-cell complex with only one $2$-cell. It can be shown that this complex can not be triangulated with less than $3m$ triangles, the boundary of the $2$-cell being mapped into the $1$-skeleton by a PL-map of degree $m$. The Moebius telescope thus shows that to get an economic simplicial complex whose fundamental group is $\Z_m$ we first have to start with a $2$-cell complex with roughly $\log_2 m$ cells of dimension $2$.
\end{remark}

\begin{remark} For $m=2^n+1$, the proof of Lemma \ref{lemme:1} implies that
$$
\kappa(\Z_m)\leq 9\log_2(m-1)+8.
$$
 For $m=3$ this upper bound is sharp (see the table in Example \ref{ex:table}) and the Moebius telescope gives thus the minimal complex for $\Z_3$. It is natural to ask whether
$\kappa(\Z_m)=9\log_2 (m-1)+8$ for $m=2^n+1$.
\end{remark}

We now prove the general upper bound of Theorem \ref{teo:cyclique}. Consider the decomposition
$$
G = \Z_{n_1} \oplus \Z_{n_2} \oplus \ldots \oplus \Z_{n_s}
$$
where $n_1|n_2|\dots |n_s$.

For $k=1,\ldots,s$ take the economic $2$-simplicial complex $X_{n_k}$ constructed in the proof of Lemma \ref{lemme:1} with fundamental group $\Z_{n_k}$. By gluing all the points $P_k \in X_{n_k}$ for $k=1,\ldots,s$ into the same point $P$, we obtain a $2$-simplicial complex
$$
\mathop{\bigvee}\limits_{k=1}^s X_{n_k}
$$
with at most
$$
\mathop{\sum}\limits_{k=1}^s (14\log_2n_k + 3) = 14\log_2|G| + 3s
$$
$2$-simplices.

The fundamental group of this wedge sum is the group
$$
\Z_{n_1}\ast\ldots\ast\Z_{n_s}.
$$
In order to get a simplicial complex whose fundamental group is $G$ we fix for each $k =1,\ldots,s$ a loop $\alpha_k \subset X_{n_k}$ based at $P$, consisting of three edges and generating the fundamental group $\pi_1(X_{n_k})\simeq \Z_{n_k}$. For each $1\leq k <l\leq s$ we glue a minimal triangulated $2$-torus to our $2$-simplicial complex by identifying the pair of loops $(\alpha_k,\alpha_l)$ with a pair of loops $(\alpha',\alpha'')$ of the minimal $2$-torus as depicted in Figure \ref{fig-trigtore}.
We thus get a new $2$-simplicial complex $X_G$ with fundamental group $G$ as each pair $(\alpha_k,\alpha_l)$ now commutes. Because we add $7 s (s-1)$ triangles ($14$ triangles for each minimal $2$-torus)  we have that
$$
s_2(X_G) \leq 14\log_2|G| + 3s + 7s(s-1) = 14\log_2|G| + 7s^2-4s,
$$
which concludes the proof of Theorem \ref{teo:cyclique}.
\end{proof}

\section{From simplicial complexity to systolic area for groups}

In this section we first prove Theorem \ref{th:finitude}. Then we show how to derive inequality (\ref{eq:syscyclique}).

\subsection{Proof of Theorem \ref{th:finitude}}
We start with the proof of the right-hand side inequality in Theorem \ref{th:finitude}. \\

By the right-hand side inequality of Theorem \ref{th:syscomplexity} we know that 
$$
\mathcal{G}_{\sigma}(T) \subset \mathcal{G}_{\kappa}\left(625 \cdot (50^2 T)^{1 + {2+\ln 25 \over \sqrt{\ln(50^2 T)}}}\right).
$$
This together with the right-hand side inequality in Theorem \ref{th:borne.sup} gives that
$$
|\mathcal{G}_{\sigma}(T)| \leq e^{6K\ln K},
$$
where $K = 625 \cdot (50^2 T)^{1 + {2+\ln 25 \over \sqrt{\ln(50^2 T)}}}$.
If $T \geq 2$, then 
$$
{2+\ln 25 \over \sqrt{\ln(50^2 T)}}\leq 2
$$
and we compute that
\begin{eqnarray*}
6K\ln K &= &6\cdot 625 \cdot (50^2 T)^{1 + {2+\ln 25 \over \sqrt{\ln(50^2 T)}}} \cdot  \ln \left(625 (50^2 T)^{1 + {2+\ln 25 \over \sqrt{\ln(50^2 T)}}} \right)\\
&\leq& {6 \cdot 625 \cdot 50^6}\cdot T^{1 + {2+\ln 25 \over \sqrt{\ln(50^2 T)}}} \cdot 4 \ln(50^2 T)\\
& \leq&{6 \cdot 625 \cdot 50^6}\cdot T^{1 + {39+\ln 25 \over \sqrt{\ln(50^2 T)}}}.  
\end{eqnarray*}
This implies the  right-hand side inequality in Theorem \ref{th:finitude} with $B= e^{6 \cdot 625 \cdot 50^6}$, $B'=39+\ln 25$ and $B'' = 50^2$.\\

We now turn to the proof of the right-hand side inequality in Theorem \ref{th:finitude}.\\

The first strategy we can try is the following. According to Theorem \ref{th:syscomplexity} we have
$$
\sigma(G)\leq \frac{\kappa(G)}{2\pi}
$$
which implies the inclusion
$$
\mathcal{A}_{\kappa}(T) \subset \mathcal{A}_{\sigma}\left({T \over 2\pi}\right).
$$
Thus for any positive $T$
$$
\left|\mathcal{A}_{\sigma}\left({T \over 2\pi}\right)\right|\geq \left|\mathcal{A}_{\kappa}(T) \right|\geq  \left[2^{{T - 3} \over 14}\right]
$$
according to Corollary \ref{cor:number.inf}.
But this lower bound is not as good as the left-hand side of the inequality stated in Theorem \ref{th:finitude}. To improve our estimate we proceed as follows.\\

We construct a metric version of the Moebius telescope. 
First consider the Moebius strip endowed with the Riemannian metric of constant curvature $1$ given by the quotient of a spherical strip by the central symmetry like in Figure \ref{fig-bandespherique}.

\begin{figure}[h]
\begin{center}
\includegraphics[width=0.35\linewidth]{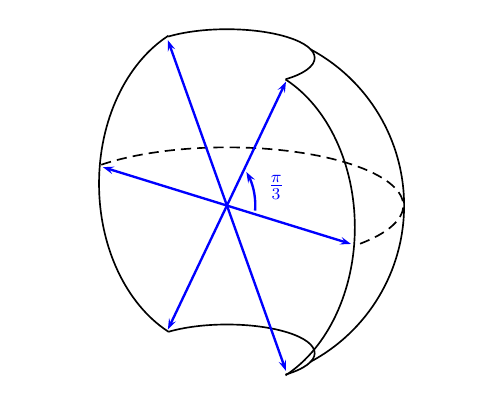}
\caption{Spherical strip.} \label{fig-bandespherique}
\end{center}
\end{figure}

The height of this spherical strip is chosen to be $2\pi/3$ and its area is thus equal to $2\pi\sqrt{3}$. The equatorial curve and the boundary have the same length $\pi$.

The metric version of the Moebius telescope is constructed in the same way as in Lemma \ref{lemme:1} using this special metric on each Moebius strip except that at each step the boundary of ${\mathcal M}_{k-1}$ is identified with the equatorial curve of ${\mathcal M}_k$ instead of with $\gamma_k$. Here $\gamma_k$ is realized as half a great circle whose intersection with the equatorial curve is a point denoted by $p_k$ and its intersection with the boundary curve another point denoted by  $q_k$. Observe that the length of $\gamma_k$ is equal to $\pi$ and that $\gamma_k$ is tangent to the boundary at $q_k$. We glue together the Moebius strips in such a way that the point $q_k$ on $\partial {\mathcal M}_k$ coincides with the point $p_{k+1}$ on the equatorial curve of ${\mathcal M}_{k+1}$ for $k=0,\ldots,n-1$.

Fix a positive integer $m$ and define $n$ to be the smallest integer such that $m < 2^{n+1}$.
The dyadic decomposition of $m$ writes $m = 2^{n_1}+\ldots+2^{n_s}$ for some integers $0 \leq n_1 < \ldots < n_s = n$. We define an analog of $\xi(m)$ by first opening up the closed curves $\gamma_{n_1}, \gamma_{n_2}, \ldots ,\gamma_{n_{s-1}},\partial \mathcal M_{n-1}$ at the points $q_{n_1}, p_{n_2}, q_{n_2}, p_{n_3}, q_{n_3},\ldots,p_{n_{s-1}}, q_{n_{s-1}}$ and $p_n$, and then connect each pair of points $(q_{n_i}, p_{n_{i+1}})$ (appearing twice) for $i=1,\ldots, s-1$ through a length minimizing arc. It may happen that $q_{n_i}= p_{n_{i+1}}$ in which case our procedure reduces to the concatenation operation at this point of the curves $\gamma_{n_i}$ and $\gamma_{n_{i+1}}$. As the distance between $p_k$ and $q_k$ is equal to $\pi/2$, we get that the curve $\xi(m)$ thus obtained is of length at most $\pi (n+1)$. We then glue a disk $D^2$ of constant curvature $1$ and radius $\frac \pi  2$ having at its center a conical singularity. The angle of the conical singularity is chosen such that the length of the boundary equals the length of the curve $\xi(m)$. We denote by $Y_m$ the resulting $2$-dimensional simplicial complex endowed with this piecewise smooth Riemannian metric.
Because the length of $\xi(m)$ is at most equal to  $\pi  (n+1)$, the area of the $2$-disk is also at most equal to $\pi (n+1)$. This implies that 
$$
\area(Y_m)\leq {(1+2\sqrt{3}) \pi} \, (n+1)\leq  {(1+2\sqrt{3}) \pi} \,(1+ \log_2m).
$$
Next we show the following.
\begin{lemma}
$\sys(Y_m)=\pi$.
\end{lemma}

\begin{proof}
Let $\gamma$ be a systolic loop, that is a non-contractible and simple closed curve of length equal to the systole. If $\gamma$ meets the conical singularity of $D^2$ then its length is at least equal to $\pi$. If $\gamma$ does not touch the conical singularity but crosses the interior of $D^2$, we can continuously deform this portion of $\gamma$ to the boundary of $D^2$ without increasing its length.

So we now assume that $\gamma \subset {\mathcal T}_n$. If $\gamma$ entirely lies inside some Moebius band ${\mathcal M}_k$ then it is straightforward to check that its length is at least equal to $\pi$. 
If not, let denote by ${\mathcal M}_k$ a Moebius band whose interior is crossed by $\gamma$ and consider a maximal non-empty connected subarc $c$  of $\gamma \cap {\mathcal M}_k$. There are three alternatives.

\noindent a) The endpoints of $c$ both lie on the equatorial curve of ${\mathcal M}_k$. Because $c$ is a geodesic arc, it is a half of a great circle and its length is at least $\pi$;

\noindent b) The endpoints of $c$ both lie on $\partial {\mathcal M}_k$. This subarc being a portion of a great circle, it necessarily cuts the equatorial curve and its length is thus  at least equal to $2\pi/3$. The remaining part of the curve must both start and end transversally to the equatorial curve of ${\mathcal M}_{k+1}$. If not, the entire curve would lie inside ${\mathcal M}_k$, case which has been excluded (in particular the case where $k=n$ is ruled out). Thus either the remaining curve is entirely contained in ${\mathcal M}_{k+1}$ in which case its length is equal to $\pi$, or its length is at least twice the distance from the equatorial curve to the boundary of ${\mathcal M}_{k+1}$, that is at least $2\pi/3$. Thus the length of $\gamma$ is always at least equal to $\pi$;

\noindent c) One of the endpoints of $c$ lies on the equatorial curve of ${\mathcal M}_k$ and the other on $\partial {\mathcal M}_k$. This subarc has length at least equal to $\pi/3$. By arguing in the same way as above, we see that the remaining part of $\gamma$ has length at least equal to $2\pi/3$ thus concluding the proof. 
\end{proof}

Because $\pi_1(Y_m)=\Z_m$, this implies that
$$
\sigma(\Z_m)\leq {1+2\sqrt{3}\over \pi} \, (1+\log_2m).
$$
We thus obtain the right-hand side of inequality (\ref{eq:syscyclique}), and the left-hand side inequality in Theorem \ref{th:finitude}:
$$
 |\mathcal{A}_{\sigma}(T)|\geq \left[2^{{\pi \over 1+2\sqrt{3}} T}\right].
$$

\subsection{Systolic area for finite abelian groups} 

For a finite abelian group $G$, the behaviour of $\sigma (G)$ in terms of the number $|G|$ of its elements  and the number of its invariant factors can be described by combining Theorems \ref{th:syscomplexity} and \ref{teo:cyclique}. First we easily derive the following.

\begin{theorem}\label{teo:sigma.abelien}
Let $G$ be a finite abelian group with $s$ invariant factors. Then
$$
\sigma (G) \leq {1\over 2\pi}\left(14 \log_2|G| + 7s^2 - 4s\right).
$$
\end{theorem}

Now recall that given a group $G$ with free index zero we have
$$
\kappa(G)\leq 625 \, (50^2 \cdot \sigma(G))^{1+{2+\ln 25 \over \sqrt{\ln (50^2 \cdot \sigma(G))}}}
$$
according to Theorems \ref{th:syscomplexity}.
We get
$$
\ln \left({\kappa(G) \over 625}\right) \leq (2+\ln 25) \, \sqrt{\ln (50^2 \cdot \sigma(G))}+\ln (50^2 \cdot \sigma(G))
$$
which implies with $C=2+\ln 25\simeq 5.2$ that
$$
\sqrt{\ln (50^2 \cdot \sigma(G))}\geq \frac{\sqrt{C^2+4 \ln \left({\kappa(G) \over 625}\right) }-C}{2}
$$
provided $\kappa(G) \geq 625 \cdot e^{-C^2\over 4} \simeq 0,69$. This condition being always fullfilled as $\kappa(G) \neq 0$, we deduce the following result.

\begin{corollary}
Let $G$ be a group with free index zero. Then
$$
\sigma (G) \geq {1 \over 50^2}\left({\kappa(G) \over 625}\right)^{1-\varphi\left({\kappa(G) \over 625}\right)},
$$
where $\varphi : [1,\infty[ \to \R$ denotes the decreasing function given by the formula
$$
\varphi(x) = {C^2\over 2 \ln x}\left(\sqrt{1+{4 \ln x \over C^2}}-1\right).
$$
\end{corollary}

Because a finite abelian group $G$ satisfies
$$
\kappa(G) \geq 2 \, \log_3 |G|,
$$
we derive the following lowerbound.

\begin{theorem}\label{teo:sigma.abelien2}
A finite abelian group $G$ satisfies the inequality
$$
\sigma (G) \geq  {1 \over 50^2} \left({2\log_3(|G|)\over 625} \right)^{1 -
\varphi\left({2\log_3(|G|)\over 625} \right)}.
$$
\end{theorem}

Finally, observe that $\varphi(x) \sim {C \over \sqrt{\ln x}}$
for large values of $x$, which leads in particular to the almost logarithmic asymptotic lower bound on $\sigma(\Z_m)$ presented in the introduction (see the left-hand side of inequality (\ref{eq:syscyclique})).

\begin{corollary}\label{cor:sigma.Z.m}
For any positive $\varepsilon$,
$$
\Sys (\Z_m)\geq (\log_2 m)^{1 - \varepsilon}
$$
provided $m$ is large enough.
\end{corollary}

In the same way we conclude that for any positive $\varepsilon$
$$
\Sys(G) \geq (\ln |\mbox{Tors}\,H_1(G, \Z)|)^{1-\varepsilon}
$$
for groups with large torsion in homology according to Proposition \ref{prop:simplexe2}.

\section{Applications to systolic volume of homology classes}\label{sec:sysvol}

In this section we first recall the definition of systolic volume associated to a homology class of a group, and then explain how to derive an interesting lower bound for the systolic volume in terms of the $1$-torsion of the group using the notion of simplicial complexity. \\

Let ${\bf a} \in H_n(G,\Z)$ be  an $n$-dimensional homology class of a group $G$ where $n$ denotes some positive integer.
A {\it geometric cycle $(X,f)$ representing the class ${\bf a}$} is a pair $(X,f)$ where
$X$ is an orientable pseudomanifold $X$ of dimension $n$
and $f : X \to K(G,1)$ a continuous map such that $f_\ast[X]={\bf a}$ where $[X]$ denotes the fundamental class of $X$ and $K(G,1)$ is 
the Eilenberg-MacLane space of $G$.
 The representation is said to be {\it normal} if in addition the
induced map $f_\sharp : \pi_1(X) \to G$ is an epimorphism. Given a
geometric cycle $(X,f)$ and a piecewise smooth metric $g$ on $X$, the  {\it relative homotopic
systole} $\sys_f(X,g)$ is defined as the least length of a loop
$\gamma$ of $X$ such that $f\circ \gamma$ is not contractible. The
{\it systolic volume} of $(X,f)$ is then the number
$$
\SSys_f(X):=\underset{g}{\inf}\, \frac{\vol(X,g)}{\sys_f(X,g)^n},
$$
where the infimum is taken over all piecewise smooth metrics $g$  on $X$
and $\vol(X,g)$ denotes the $n$-dimensional volume of $X$. When $f
: X \to K(\pi_1(X),1)$ is the classifying map (induced by an
isomorphism between the fundamental groups), we simply denote by
$\SSys(X)$ the systolic volume of the pair $(X,f)$.  From
\cite[Section 6]{Gro83} we know that for any  dimension $n$
$$
\SSys_n:= \underset{(X,f)}{\inf}\, \SSys_f(X) >0,
$$
the infimum being taken over all geometric cycles $(X,f)$
representing a non trivial homology class of dimension  $n$. The
following notion was introduced by Gromov in \cite[Section
6]{Gro83}.

\begin{definition}\label{def:sys}
The {\it systolic volume} of a homology class ${\bf a} \in H_n(G,\Z)$ is defined as the
number
$$
\SSys({\bf a}):=\underset{(X,f)}{\inf} \, \SSys_f(X)
$$
where the infimum is taken over all geometric cycles $(X,f)$
representing the class ${\bf a}$.
\end{definition}
Observe that for any homology class ${\bf a} \in H_2(G,\Z)$ we have
$$
\SSys({\bf a})\geq \sigma(G).
$$
Recently the systolic volume of homology classes has been extensively studied in \cite{BB15} where the reader can find numerous results on this invariant.

\subsection{A lower bound of systolic volume by $1$-torsion}
We now define the  $1$-torsion of a homology class and explain how to use it to bound from below its systolic volume.
\begin{definition}\label{def:1tor}
The {\it $1$-torsion of a homology class ${\bf a}  \in H_n(G,\Z)$} is defined as the
number
$$
t_1({\bf a})=  \underset{(X,f)}{\inf} |\mbox{Tors}\,H_1(X, \Z)|
$$
where the infimum is taken over the set of geometric cycles $(X,f)$
representing the class  ${\bf a}$ and $|\mbox{Tors}\,H_1(X, \Z)|$
denotes the number of torsion elements in the first integral
homology group of  $X$.
\end{definition}

 We now present the main result of this
section.

\begin{theorem}\label{th:torsion}
Let $G$ be a group and ${\bf a} \in
H_n(G,\Z)$. Then
$$
\SSys({\bf a}) \geq C_n\big{(}\ln t_1({\bf a})\big{)}^{1 -
\frac{C'_n}{\sqrt{\ln (\ln t_1({\bf a})})}} ,
$$
where $C_n$ and $C'_n$ are two positive numbers depending only on
$n$.
\end{theorem}

In particular,
$$
\SSys({\bf a})\geq (\ln t_1({\bf a}))^{1-\varepsilon}
$$
for any $\varepsilon>0$ provided $t_1({\bf a})$ is large enough. It is
important to remark that there is no hope in dimension $\geq 3$  to
prove a universal lower bound of the type
$$
\SSys({\bf a})\geq C \ln t_1({\bf a})
$$
for some positive constant $C$. Indeed, for any positive integer
$s$, the Eilenberg-MacLane space of the group
$$
G_s:=\underset{s}{\underbrace{\Z_2 \ast \ldots \ast \Z_2}}
$$
is the complex $\bigvee_{i=1}^s \R P^\infty_i$. If  $n$ is a positive integer and $\R P^{2n+1}_i
\subset \R P^\infty_i$ denotes the skeleton of odd dimension
$2n+1$ of the $i$-th component, we consider the sequence of
homology classes
$$
{\bf a}_s=\sum_{i=1}^s[\R P^{2n+1}_i] \in H_{2n+1}(G_s,\Z).
$$
We can see that $t_1({\bf a}_s)=2^s$ as $|\mbox{Tors}\,H_1(X, \Z)|\geq 2^s$ for any
representation $(X,f)$ of ${\bf a}$. According to
 \cite[Theorem 5.4]{BB15} this implies that
$$
\SSys({\bf a}_s)\leq \SSys(\#_{i=1}^s \R P^{2n+1}_i)\leq C
\cdot \frac{\ln t_1({\bf a}_s)}{\ln \ln t_1({\bf a}_s)}
$$
for some positive constant $C$. For even dimension, we consider
the sequence of classes $\tilde{{\bf a}}_s=[S^1] \times {\bf a}_s
\in H_{2n+2}(\Z \times G_s,\Z)$
for which the same upper bound holds.  \\

\begin{proof}[Proof of Theorem \ref{th:torsion}] Let $G$ be a finitely
presentable group and ${\bf a} \in H_n(G,\Z)$ a homology class. Recall
that the {\it simplicial height} $h({\bf a})$ of a homology class
${\bf a}$ is the minimum number of simplexes (of any dimension)
of a simplicial complex representing the class ${\bf a}$.
According to \cite[6.4.C'']{Gro83} (see also \cite[3.C.3]{Gro96})
there exists two positive numbers $c_n$ and $c'_n$ depending only on the
dimension $n$ such that
$$
\SSys({\bf a}) \geq c_n \cdot \frac{h({\bf
a})}{\exp(c'_n\sqrt{\ln h({\bf a})})}.
$$
We conclude using Proposition \ref{prop:simplexe2} which asserts that
$$
h({\bf a}) \geq 2 \, \log_3 \, t_1({\bf a}).
$$
 \end{proof}

\subsection{Application to lens spaces}
In general the $1$-torsion of a class is difficult to compute. In
the case of $G = \Z_m$, we can estimate from below the $1$-torsion
of any generator by the number $m$ as follows.

\begin{lemma}\label{lemma:representation}
Let ${\bf a}$ be a generator of $H_{2n+1}(\Z_m, \Z)$. Then
$$
t_1({\bf a})\geq m.
$$
\end{lemma}
\begin{proof}
Let $(X,f)$ be a geometric cycle representing ${\bf a}$. As ${\bf
a}$ is a generator of $H_{2n+1}(\Z_m, \Z) \simeq \Z_m$, the map
$f$ induces an isomorphism
\begin{equation}\label{eq:isom}
f^* : H^{2n+1}(\Z_m, \Z_m) \to H^{2n+1}(X, \Z_m).
\end{equation}
Let
$$
\beta : H^1(\Z_m, \Z_m) \to H^2(\Z_m, \Z) .
$$
denotes the  Bockstein homomorphism and
$$
j: H^2(\Z_m, \Z)\to H^2(\Z_m, \Z_m)
$$
the morphism of reduction modulo $m$. In our case, $j$ is an
isomorphism. A generator of $H^{2n+1}(\Z_m, \Z_m)$ (not
necessarily dual to ${\bf a}$) can be choosen as the element ${\bf u} \cup (j
\circ \beta({\bf u}))^n$ where  ${\bf u} \in H^1(\Z_m, \Z_m)$ is
some generator. Now consider $f^*(\beta({\bf u})) = \beta(f^*({\bf
u})) \in H^2(X, \Z)$. Taking into account the isomorphism
(\ref{eq:isom}) we see that $f^*({\bf u}) \cup \big{(}f^*(j \circ \beta({\bf
u}))\big{)}^n$ is an element of order $m$ in $H^{2n+1}(X, \Z_m)$.
This implies that the order of $f^*(j \circ \beta({\bf u})) \in
H^2(X, \Z_m)$ is $m$, and thus that the order of $\beta(f^*({\bf u})) \in
H^2(X,\Z)$ is also $m$. By the duality (\ref{eq:dual}) we get the
result.
\end{proof}

Remark that the statement of this lemma as well as its proof holds
in the case of a simplicial complex $X$ representing the class
${\bf a}$. In this more general case note that
the map (\ref{eq:isom}) is only injective. \\

Given two integers $n\geq 0$ and $m \geq 2$ let $L_n(m)$ denote a
lens space of dimension $2n+1$ with fundamental group $\Z_m$: there exist integers $q_1, \ldots,q_m$ coprime with $m$ and an
isometry $A$ of order $m$ of the form
$$
A(z_1,\ldots,z_n) = (e^{2\pi i {q_1 \over m}}z_1, \ldots, e^{2\pi
i {q_n \over m}}z_n)
$$
such that
$$
L_n(m):=\{Z=(z_1,\ldots,z_n) \in \C^n \mid |z_1|^2+\ldots+|z_n|^2=1\}/_{\sim A} \, \, \simeq S^{2n+1}/\Z_m,
$$
where  $Z\sim Z'$ if and only if $Z = A^k \,Z'$. Observe that the
fundamental class of a lens space $L_n(m)$ realizes a generator
${\bf a}$ of the homology group $H_{2n+1}(\Z_m, \Z)$.

Combining Lemma \ref{lemma:representation} and Theorem
\ref{th:torsion} we derive the following result.

\begin{theorem}\label{th:lenticulaire}
For any integer $m\geq 2$ we have
$$
\SSys(L_n(m)) \geq C_n\big{(}\ln m\big{)}^{1 -
\frac{C'_n}{\sqrt{\ln (\ln m)}}}
$$
where $C_n$ and $C'_n$ are two positive numbers depending only on the dimension $n$.
\end{theorem}

We remark that this lower bound is of
the same type as that for $\sigma(\Z_m)$ in Corollary
\ref{cor:sigma.Z.m}. Now we turn to the proof of the new upper bound for systolic volume of lens spaces.

\begin{theorem}\label{th:lenticulaire2}
For any integer $m\geq 2$ we have
$$
\SSys(L_n(m)) \leq D_n m^n
$$
where $D_n$ is a positive number depending only on the dimension $n$.
\end{theorem}

As observed in the introduction, this polynomial lower bound is better than the one obtained by computing the systolic volume for the round metric (which is roughly $ \approx m^{2n}$).

\begin{proof}
We first decompose the manifold $L_n(m)$ into $(2n+1)$-dimensional cubes, and then use this decomposition to construct a metric for which we control the systolic volume in terms of the number of these cubes, compare with \cite{BB05} and \cite{BB15}.\\

We decompose $L_n(m)$ into $(2n+1)$-cubes as follows. Start with the standard cellular decomposition of $L_n(m)$ denoted by $\Theta$  (see  \cite[p.145]{Hat02} generalizing for the case  $n>1$ the construction of \cite{ST80}). This decomposition $\Theta$ has exactly one $k$-cell denoted by $e^k$ in each dimension $k=0, 1,\ldots , 2n+1$. Denote by $q_k$ its center. We subdivide the cellular decomposition $\Theta$ in such a way that each new cell admits a simplicial structure. We proceed by induction as follows.
For $k=0$ there is nothing to do but observe that $e^0=\{q_0\}$. Because $e^1$ is attached to $e^0$ we subdivide $e^1$ into the two arcs denoted by $e^1_0$ and $e^1_1$ connecting $q_0$ and $q_1$. Observe at this stage that the complex thus obtained is not a simplicial complex, as two simplices may share more than one face in common. The next step consists to first remark that $e^2$ is attached to $e^1\cup e^0$ through a linear map $\varphi_2: \partial e^2 \to e^1 \cup e^0$ of degree $m$, and then take in $e^2$ the cone over the preimages by $\varphi_2$ of $e^1_0$ and $e^1_1$ with respect to the vertex $q_2$. Doing so we subdivise $e^2$ into $2m$ new $2$-cells with a natural structure of a simplex.
We then proceed in that way by induction on the dimension following the structure of $\Theta_1$. More precisely at each step $k\geq 3$ we form the cone over the preimages by the attaching map of the $k-1$-dimensional new cells with respect to the center of the $k$-cell. The attaching maps
$$
\varphi_k: \partial e^k \longrightarrow \mathop{\bigcup}\limits_{s=0}^{k-1} e^s
$$
having degree $m$ if $k$ is even, and zero if $k$ is odd, this gives a new decomposition $\Theta_1$ with
$2^{n+1}m^n$ simplices of dimension $2n+1$. Remark that despite the fact that $\Theta_1$ is not a simplicial decomposition this structure is coherent in the sense that any face of a simplex is a simplex of lower dimension.

Denote by $\Theta_2$ the barycentric subdivision of $\Theta_1$. The structure $\Theta_2$ is now simplicial with $2^{n+1}\cdot (2(n+1))!\cdot m^n$ simplices of dimension $2n+1$.
We decompose each ($2n+1$)-simplex of $\Theta_2$ into $2(n+1)$ cubes of dimension $2n+1$. This gives  a decomposition $\Theta_3$ of $L_n(m)$ into $2^{n+1}\cdot (2(n+1))! \cdot 2(n+1)\cdot m^n$ cubes of
dimension $2n+1$.
Endow  each cube of $\Theta_3$ with the Euclidean metric with side length $1$. We thus get a polyhedral metric $g$ on $L_n(m)$. This metric satisfies  $\sys(L_n(m), g) \geq 2$ according to \cite[Lemma 5.6]{BB15}. Because $\vol(L_n(m), g) =2^{n+1}[2(n+1)]![2(n+1)] m^n$, this gives the result with $D_n =
{[2(n+1)]!(n+1) \over 2^{n-1}}$.
\end{proof}

\subsection{Application to $3$-manifolds}

Let $M$ be a closed manifold. For a covering space $M'$ with $k$
sheets of $M$, it is straightforward to check that
\begin{equation}\label{eq:revetement}
\SSys(M) \geq {1 \over k} \SSys(M') .
\end{equation}
Let $M$ be a manifold of dimension $3$ with finite fundamental
group. Its universal cover is the sphere $S^3$, and the action of
the fundamental group $\pi_1(M)$ on $S^3$ is orthogonal. The list
of finite groups which act orthogonally on $S^3$ can be found in
\cite{Miln57} for instance. An analysis of this list shows that
$\pi_1(M)$ possesses a cyclic subgroup of index $k \leq 12$.
Denote by $M'$ the covering space corresponding to this subgroup.
The manifold  $M'$ is a lens space  $L_1(n)$ with $n \geq
{|\pi_1(M)| \over 12}$. So by applying Theorem
\ref{th:lenticulaire} and the inequality (\ref{eq:revetement})
with $k = 12$ we derive the last theorem of this paper.

\begin{theorem}\label{th:varietes3dim}
There exists two positive constants  $a$ and $b$ such that any
manifold  $M$ of dimension $3$ with finite fundamental group
satisfies
$$
\SSys(M) \geq a\big{(}\ln |\pi_1(M)|\big{)}^{1 -
\frac{b}{\sqrt{\ln (\ln |\pi_1(M)|)}}} ,
$$
where $|\pi_1(M)|$ denotes the number of elements in $\pi_1(M)$.
\end{theorem}

Remark that finite fundamental groups of  $3$-manifolds can have a
large number of elements but a very small torsion in
$H_1(\pi_1(M), \Z)$. A direct estimate of $\SSys(M)$ by the
torsion of $H_1(\pi_1(M), \Z)$ is interesting only if the manifold
$M$ is a lens space.\\

\noindent {\bf Acknowledgements.} The authors are grateful to the referees whose valuable comments lead us to considerably improve the presentation of this article, and also to one of the referees who pointed out a mistake in the initial proof of Theorem \ref{th:syscomplexity}.


\begin{thebibliography}{99}

\bibitem[Bab06]{Bab06}
Babenko, I.: {\em Topologie des systoles unidimensionnelles}.
Enseign. Math. {\bf 52} (2006), 109-142.


\bibitem[BB05]{BB05}
Babenko, I. \& Balacheff, F.: {\it G\'eom\'etrie systolique des
sommes connexes et des rev\^etements cycliques}. Math. Ann. {\bf
333} (2005), 157-180.


\bibitem[BB15]{BB15}
Babenko, I. \& Balacheff, F.: {\em Systolic volume of homology classes}.
Alg. Geom. Top., {\bf 15} (2015), 733-767. 


\bibitem[BPS12]{BPS12}
Balacheff, F., Parlier, H. \& Sabourau, S.: {\it Short loop
decompositions of surfaces}. Geom. Funct. Anal. {\bf 22} (2012), 37-73.

\bibitem[BS10]{BS10}
Balacheff, F. \& Sabourau, S.: {\it Diastolic and isoperimetric inequalities on surfaces}.
Ann. Sc. de l'ENS, {\bf 43} (2010), 579-605


\bibitem[Bul14]{Bul14} Bulteau, G.: {\em Les groupes de petite complexitï¿½\'e simpliciale}, preprint.

\bibitem[BS94]{BS94}
Buser, P. \& Sarnak, P.: {\it On the period matrix of a Riemann surface of large genus. With an appendix by J. H. Conway and N. J. A. Sloane.}
Invent. Math. {\bf 117} (1994), 27-56.

\bibitem[Cha15]{Cha15}
Cha, J.C.: {\it Complexities of 3-manifolds from triangulations, Heegaard splittings, and surgery presentations}. Preprint arXiv https://arxiv.org/abs/1506.00757.

\bibitem[Cha16]{Cha16} 
Cha, J.C.:  {\it Complexity of surgery manifolds and Cheeger-Gromov invariants}. 
Int. Math. Res. Not. {\bf 18} (2016), 5603-5615.



\bibitem[Del96]{Del96} Delzant, T.: {\em D\'ecomposition d'un groupe en produit libre ou somme amalgam\'ee}.
J. Reine Angew. Math. {\bf 470}(1996), 153-180.

\bibitem[Gro83]{Gro83}
Gromov, M. : {\it Filling Riemannian manifolds}.
J. Diff. Geom. {\bf 18} (1983), 1-147.

\bibitem[Gro96]{Gro96}
Gromov, M.: {\em Systoles and intersystolic inequalities}.
Actes de la {T}able {R}onde de {G}\'eom\'etrie {D}iff\'erentielle, Collection SMF {\bf 1} (1996), 291-362.

\bibitem[Hat02]{Hat02}
Hatcher, A. : {\it Algebraic Topology}. Cambridge University Press
(2002).

\bibitem[JRT09]{JRT09} 
Jaco, W., Rubinstein, H. \& Tillmann, S.: {\it Minimal triangulations for an infinite family of lens spaces}.
J. Topol. {\bf 2} (2009), 157-180.
 
\bibitem[JRT11]{JRT11} 
Jaco, W., Rubinstein, H. \& Tillmann, S.: {\it Coverings and minimal triangulations of $3$-manifolds}.
Algebr. Geom. Topol. {\bf 11} (2011), 1257-1265.

\bibitem[JRT13]{JRT13}
Jaco, W., Rubinstein, H. \& Tillmann, S.:  {\it $\Z_2$-Thurston norm and complexity of $3$-manifolds}.
Math. Ann. {\bf 356} (2013), 1-22.

\bibitem[JR80]{Ring}
Jungerman, M. \& Ringel, G. : {\it Minimal triangulations on orientable surfaces}.
Acta Math.  {\bf 145}  (1980), 121-154.

\bibitem[KS05]{KS05} Kapovich, I. \& Schupp, P.: {\em Delzant's T-invariant, one-relator groups, and Kolmogorov's complexity}.
Com. Math. Helv. {\bf 80}ï¿½(2005), 911-933.

\bibitem[Kur60]{Kur60} Kurosh, A.G.: {\em The Theory of Groups}.
Chelsea Pub. Company, N.Y. 1960.

\bibitem[Mat07]{Mat07}
Matveev, S.: {\it Algorithmic topology and classification of 3-manifolds}.
Second edition. Algorithms and Computation in Mathematics, 9. Springer, Berlin, 2007


\bibitem[Mas67]{Mas67}
Massey, W.: {\em Algebraic Topology : An Introduction}.
Harcourt, Brace $\&$ World, Inc. 1967.

\bibitem[MP01]{MP01} Matveev, S. \& Pervova, E.: {\em Lower bounds for the complexity of three-dimensional manifolds}.
Dokl. Akad. Nauk. {\bf 378}ï¿½(2001), 151-152.

\bibitem[MPV09]{MPV09} 
Matveev, S., Petronio, C. \& Vesnin, A.:Ê{\it Two-sided asymptotic bounds for the complexity of some closed hyperbolic three-manifolds}.
J. Aust. Math. Soc. {\bf 86} (2009), 205-219.

\bibitem[Mil57]{Miln57}
Milnor,  J.: {\em Groups which act on $S^n$ without fixed points}.
Amer. J. Math. {\bf 79} (1957), 623-630.

\bibitem[PP08]{PP08} Pervova, E. \& Petronio, C.: {\em Complexity and T-invariant of Abelian and Milnor Groups, and complexity of 3-manifolds}.
Math. Nachr. {\bf 281} (2008), 1182-1195.

\bibitem[Pu52]{Pu52}
Pu, P.: {\it Some inequalities in certain non-orientable Riemannian manifolds}.
Pacific J. of Math. \textbf{2}  (1952), 55-71.



\bibitem[RS08]{RS08}
Rudyak, Y. \& Sabourau, S.: {\em Systolic invariants of groups and 2-complexes via Grushko decomposition}.
Ann. Inst. Fourier {\bf 58} (2008), 777-800.


\bibitem[ST80]{ST80}
Seifert, H. \& Threlfall, W.:  {\em A textbook of
topology}, Academic Press (1980).

\bibitem[VdW71]{VdW71}
Van der Waerden, B. L. : {\it Algebra 1}.
Achte auflage der modernen algebra, Springer Verlag (1971).



\end{thebibliography}
\end{document}